\PassOptionsToPackage{dvipsnames}{xcolor}
\documentclass[VANCOUVER,STIX1COL]{WileyNJD-v2}
\usepackage{moreverb}
\usepackage{amsfonts,amssymb}
\usepackage{graphicx}
\graphicspath{{./Figures/}}
\usepackage{color}

\usepackage{ulem}
\usepackage{multirow}
\usepackage{multicol}
\usepackage[showonlyrefs]{mathtools}
\usepackage{subcaption}

\usepackage{url}
\DeclareMathAlphabet{\mathbf}{OT1}{cmr}{bx}{it}

\newcommand{\vb}{{\mathbf b}}

\newcommand{\ve}{{\mathbf e}}

\newcommand{\vf}{{\mathbf f}}

\newcommand{\vu}{{\mathbf u}}
\newcommand{\vutilde}{{\mathbf {\widetilde{u}}}}
\newcommand{\vv}{{\mathbf v}}
\newcommand{\vvhat}{{\mathbf {\hat v}}}
\newcommand{\vwhat}{{\mathbf {\hat w}}}

\newcommand{\vw}{{\mathbf w}}
\newcommand{\vx}{{\mathbf x}}
\newcommand{\vxtilde}{{\mathbf {\widetilde{x}}}}

\newcommand{\vy}{{\mathbf y}}
\newcommand{\vz}{{\mathbf z}}
\newcommand{\vnull}{\boldsymbol{0}}

\newcommand{\spK}{{\cal K}}
\newcommand{\spQ}{{\cal Q}}

\renewcommand{\d}{\,\mathrm{d}}
\newcommand{\dmu}{\d\mu(t)}

\newcommand{\N}{\mathbb{N}}
\newcommand{\R}{\mathbb{R}}

\newcommand{\Rnn}{\mathbb{R}^{n \times n}}
\newcommand{\C}{\mathbb{C}}
\newcommand{\bbE}{\mathbb{E}}
\newcommand{\Cn}{\mathbb{C}^n}
\newcommand{\Cnn}{\mathbb{C}^{n \times n}}

\newcommand{\lmax}{\lambda_{\textnormal{max}}}
\newcommand{\lmin}{\lambda_{\textnormal{min}}}

\newcommand{\That}{\widehat{T}}
\newcommand{\Lhat}{\widehat{L}}
\newcommand{\Vhat}{\widehat{V}}
\newcommand{\What}{\widehat{W}}
\newcommand{\Tcal}{\mathcal{T}}

\newcommand{\tildeV}{\widetilde{V}_m}
\newcommand{\tildeT}{\widetilde{T}_m}
\newcommand{\Vcal}{\mathcal{V}}

\newcommand{\Lan}{\textnormal{Lan}}
\newcommand{\Arn}{\textnormal{Arn}}

\newcommand{\Rat}{\textnormal{Rat}}
\newtheorem{example}[theorem]{Example}
\let\oldexample\example
\renewcommand{\example}{\oldexample\normalfont}

\DeclareMathOperator{\Span}{span}
\DeclareMathOperator{\spec}{spec}

\newtheorem{remarksimple}[theorem]{Remark}

\def\Fd{Fr\'{e}chet derivative}

\usepackage{pgfplots}
\usepgfplotslibrary{groupplots}
\usetikzlibrary{external}

\pgfkeys{/pgf/images/include external/.code={\includegraphics{#1}}}

\newenvironment{customlegend}[1][]{%
    \begingroup
    \csname pgfplots@init@cleared@structures\endcsname
    \pgfplotsset{#1}%
}{%
    \csname pgfplots@createlegend\endcsname
    \endgroup
}%
\def\addlegendimage{\csname pgfplots@addlegendimage\endcsname}

\definecolor{color_peter1}{RGB}{0,74,255}
\definecolor{color_peter2}{RGB}{0,200,0}
\definecolor{color_peter3}{RGB}{255,147,49}

\newcommand\BibTeX{{\rmfamily B\kern-.05em \textsc{i\kern-.025em b}\kern-.08em
T\kern-.1667em\lower.7ex\hbox{E}\kern-.125emX}}

\articletype{RESEARCH ARTICLE}%

\received{<day> <Month>, <year>}
\revised{<day> <Month>, <year>}
\accepted{<day> <Month>, <year>}

\begin{document}

\title{Computing low-rank approximations of the Fr\'echet derivative of a matrix function using Krylov subspace methods}

\author[1]{Peter Kandolf}

\author[2]{Antti Koskela}

\author[3]{Samuel D. Relton}

\author[4]{Marcel Schweitzer*}

\authormark{KANDOLF \textsc{et al}}

\address[1]{\orgdiv{Institut f\"ur Mathematik}, \orgname{Universit\"at Innsbruck}, \orgaddress{\state{Technikerstr.\ 13, 6020 Innsbruck}, \country{Austria}}}

\address[2]{\orgdiv{Department of Computer Science}, \orgname{University of Helsinki}, \orgaddress{\state{P.O.Box 68, FIN-00014}, \country{Finland}}}

\address[3]{\orgdiv{Leeds Institute of Health Sciences}, \orgname{The University of Leeds}, \orgaddress{\state{Leeds, LS2 9LU}, \country{UK}}}

\address[4]{\orgdiv{Mathematisch-Naturwissenschaftliche Fakult\"at}, \orgname{Heinrich-Heine-Universit\"at D\"usseldorf}, \orgaddress{\state{Universit\"atsstra\ss{}e 1, 40225 D\"usseldorf}, \country{Germany}}. The work of Marcel Schweitzer was partly supported by the SNSF research project \emph{Low-rank updates of matrix functions and fast eigenvalue solvers}.}

\corres{*Marcel Schweitzer, Heinrich-Heine-Universit\"at D\"usseldorf, Universit\"atsstra\ss{}e 1, 40225 D\"usseldorf, Germany. \email{marcel.schweitzer@hhu.de}}


\jnlcitation{\cname{%
\author{P. Kandolf}, 
\author{A. Koskela}, 
\author{S.D. Relton}, 
and 
\author{M. Schweitzer}} (\cyear{2020}), 
\ctitle{Computing low-rank approximations of the Fr\'echet derivative of a matrix function using Krylov subspace methods}, \cvol{2020}}

\abstract[Abstract]{The \Fd\ $L_f(A,E)$ of the matrix function $f(A)$
  plays an important role in many different applications, including
  condition number estimation and network analysis. We present several
  different Krylov subspace methods for computing low-rank
  approximations of $L_f(A,E)$ when the direction term $E$ is of rank
  one (which can easily be extended to general low rank). We analyze
  the convergence of the resulting method for the important special
  case that $A$ is Hermitian and $f$ is either the exponential, the
  logarithm or a Stieltjes function. In a number of numerical tests,
  both including matrices from benchmark collections and from
  real-world applications, we demonstrate and compare the accuracy and
  efficiency of the proposed methods.}


\maketitle

\section{Introduction}\label{sec:introduction}
Matrix functions
$f \colon \mathbb{C}^{n \times n} \rightarrow \mathbb{C}^{n \times n}$
are an increasingly important part of applied mathematics with a
wide variety of applications.
The matrix exponential, $f(A) = e^A$,
arises in
network analysis~\cite{EstradaHigham2010} and
exponential integrators~\cite{HochbruckLubich1997,HochbruckOstermann2010,
  HochbruckLubichSelhofer1998};
whilst the matrix logarithm, $f(A) = \log(A)$,
occurs in models of bladder carcinoma~\cite{gsrp14}
and when computing the matrix geometric mean~\cite{jvv12}.

Also of importance is the \Fd\ of a matrix function,
defined as the unique operator
$L_f(A, \cdot)\colon \mathbb{C}^{n\times n} \rightarrow \mathbb{C}^{n\times n}$
that is linear in its second argument and,
for any matrix $E \in \mathbb{C}^{n \times n}$,
satisfies
\begin{equation}\nonumber
  \label{eq.FD_defn}
  f(A + E) - f(A) = L_f(A,E) + o(\|{E}\|),
\end{equation}
where $\|\cdot\|$ denotes the matrix two-norm and $o(\|E\|)$ represents a remainder term that,
when divided by $\|E\|$, tends to zero as
$\|E\| \rightarrow 0$.
For small-scale matrices and analytic functions $f$,
a simple way to compute the \Fd\ is via the relation
(see \cite[Thm.~2.1]{math96})
\begin{equation}\label{eq:2x2block}
f\left(\left[\begin{array}{cc}A & E \\ 0 & A\end{array}\right]\right) = \left[\begin{array}{cc} f(A) & L_f(A,E) \\ 0 & f(A)\end{array}\right].
\end{equation}
As this formula requires the evaluation of a function of a
$2n \times 2n$ matrix
(which will typically result in a dense matrix),
it is not feasible for large, sparse matrices.

The \Fd\ is primarily used to calculate the relative condition number
of computing $f(A)$ via the formula~\cite[Chap.~3]{Higham2008}.
\begin{equation}\nonumber
  \label{eq.FD_cn_defn}
  \mathrm{cond}(f, A) =
  \lim_{\epsilon \rightarrow 0} \sup_{\|E\| \le \epsilon\|A\|}
  \frac{\|L_f(A,E)\|}{\|f(A)\|}.
\end{equation}
However,
in recent years the \Fd\ has also been required in applications including
nuclear activation~\cite{amrh15}, complex network analysis~\cite{eshh09}, decomposition of tensor grids~\cite{IannazzoJeurisPompili2019} and when solving optimization problems involving matrix functions~\cite{ThanouDongKressnerFrossard2017}.

Recently Kandolf and Relton~\cite{KandolfRelton2016} proposed
a block Krylov method to form approximations of
$L_f(A,E)\vb$ where $E = \eta\vy\vz^H$ is of rank one.
They found that even for large matrices $A$ and $E$ the
\Fd\ multiplied by a vector could be computed accurately
within very few iterations. The primary goal of this work is to extend their work
to approximate the entire matrix $L_f(A,E)$ with a
low-rank representation:
this allows us to compute not only $L_f(A,E)\vb$ for multiple
vectors $\vb$,
but to also speed up the applications mentioned above.

The remainder of this work is presented as follows.  In
section~\ref{sec:frechet_approx}, we first present a general framework
for computing low-rank updates of the \Fd\ for
functions which are either represented via the Cauchy integral formula
or belong to the class of Stieltjes functions. We then give details
for various specific methods arising from this framework, depending on
the properties of $A$ and on the subspaces used. The convergence of
the resulting methods is analyzed in section~\ref{sec:convergence} for
the case that $A$ is Hermitian positive or negative
definite. Section~\ref{sec:logarithm} deals with applying the
discussed techniques to the matrix logarithm, which does not fit into
this framework. The computation of a posteriori error estimates that 
can be used as stopping criteria is covered in Section~\ref{sec:posteriori}.
In section~\ref{sec:experiments} we
perform a battery of numerical experiments to test the accuracy and
performance of our new algorithms on problems taken from benchmark
collections and real-world applications. Finally we present some
conclusions and ideas for future work in
section~\ref{sec:conclusions}.

\section{Approximating the Fr\'echet derivative}\label{sec:frechet_approx}
In this section, we show how Krylov subspace methods can be used for
constructing low-rank approximations of the \Fd\
$L_f(A,E)$ of a matrix function. We first introduce a general
framework for this without going into algorithmic details and
afterwards discuss various possible choices of specific Krylov
methods.

In the following, we assume that the direction matrix $E$ is
of rank one, i.e., $E = \eta \vy\vz^H$, where $\eta \in \C$ and
$\vy,\vz \in \Cn$. Our approach can be extended to direction matrices
of higher rank either by using the linearity of the \Fd\ with respect to $E$ (i.e., by separately applying the
method several times to rank 1 direction terms), or by using block Krylov subspace methods.

One of the main tools we use---both for the derivation of algorithms and for their convergence analysis---is an integral representation of the \Fd, which can be derived in cases where the function $f$ itself admits an integral representation involving a resolvent.

In the following, we therefore focus on two classes of functions which arise frequently in applications. The first class consists of \emph{analytic functions represented via the Cauchy integral formula}, i.e.,
\begin{equation}\label{eq:cauchy_integral}
f(A) = \frac{1}{2\pi i} \int_\Gamma f(t)(tI-A)^{-1}\d t
\end{equation}
where $\Gamma$ is a path in the complex plane that winds around
$\spec(A)$, the spectrum of $A$, exactly once. The most prominent and widely used function belonging to this class is the matrix exponential
$f(A) = \exp(A)$.
The second class we consider is the class of \emph{Stieltjes functions},
which are defined by the integral transform
\begin{equation}\label{eq:stieltjes_integral}
f(A) = \int_0^\infty (A+tI)^{-1}\d\mu(t)
\end{equation}
where $\mu$ is a nonnegative, monotonically increasing function satisfying
\begin{equation}\label{eq:stieltjes_condition}\nonumber
\int_0^\infty \frac{1}{1+t} \dmu < \infty.
\end{equation}
and we assume $\spec(A)$ $\cap$ $\R_0^- = \emptyset$. Examples of practically relevant functions belonging to this class are the inverse fractional powers $f(A) = A^{-\sigma}$, $\sigma \in (0,1)$ represented as
$$A^{-\sigma} = \int_0^\infty t^{-\sigma}(A+tI)^{-1} \d t,$$
which occur, e.g., in the solution of fractional differential equations~\cite{BurrageHaleKay2012}, in lattice quantum chromodynamics~\cite{BlochEtAl2007,VanDenEshofFrommerLippertSchillingVanDerVorst2002} or in statistical sampling~\cite{IlicTurnerPettitt2004}.

To avoid unnecessary repetition, we derive the Krylov subspace approximation for the \Fd\ only for functions of the form~\eqref{eq:cauchy_integral}, and mention that the case~\eqref{eq:stieltjes_integral} can be handled analogously, with obvious modifications.

Differentiating~\eqref{eq:cauchy_integral} using the chain rule, one finds the representation
\begin{equation}\label{eq:frechet_derivative_integral}
L_f(A,\eta\vy\vz^H) = \frac{\eta}{2\pi i} \int_\Gamma f(t)(t I - A)^{-1}\vy\vz^H(tI - A)^{-1}\d t
\end{equation}
for the \Fd\ (see~\cite{Higham2008,KandolfRelton2016}). Using the short-hand notations
\begin{equation}\label{eq:shifted_systems_solutions}
\vx(t) = (tI-A)^{-1}\vy \quad \text{and} \quad \vu(t) = (tI-A)^{-H}\vz
\end{equation}
for the solutions of the (shifted) linear systems in the integrand
of~\eqref{eq:frechet_derivative_integral},
we can write this compactly as
\begin{equation}\label{eq:frechet_derivative_integral_compact}
L_f(A,\eta\vy\vz^H) = \frac{\eta}{2\pi i} \int_\Gamma f(t)\vx(t)\vu(t)^H \d t.
\end{equation}
An approximation for $L_f(A,\eta\vy\vz^H)$ can now be found by replacing the exact solutions $\vx(t),\vu(t)$ of the shifted linear systems by approximate solutions
\begin{equation}\label{eq:shifted_systems_solutions_approx}
\vxtilde(t) \approx (tI-A)^{-1}\vy \quad \text{and} \quad \vutilde(t) \approx (tI-A)^{-H}\vz.
\end{equation}
There are several important things to consider when choosing the specific approximations to use in~\eqref{eq:shifted_systems_solutions_approx}: it should be possible to easily evaluate the integral
\begin{equation}\label{eq:frechet_derivative_integral_compact_approx}
\widetilde{L} = \frac{\eta}{2\pi i} \int_\Gamma f(t)\vxtilde(t)\vutilde(t)^H \d t
\end{equation}
without needing to choose a contour $\Gamma$ and use a numerical
quadrature rule,
and the resulting matrix $\widetilde{L}$ should be of low rank, as it will in general be a full matrix which is impossible to store
explicitly for larger values of $n$. Approximations chosen from
Krylov subspaces are natural candidates for the
approximations~\eqref{eq:shifted_systems_solutions_approx} as Krylov
subspace methods are among the most widely used methods for solving
shifted linear systems; additionally, it is well-known that the same
Krylov subspace $\spK_m(A,\vy)$ can be used for efficiently
approximating $(tI-A)^{-1}\vy$ for all values of $t$
(see, e.g.,~\cite{FrommerMaass1999,Simoncini2003}).

In the following, we discuss various choices of Krylov subspace approximations for~\eqref{eq:shifted_systems_solutions_approx} and their computational and theoretical implications.

\subsection{Lanczos approximation for Hermitian $A$ and $E$}\label{subsec:lanczos_hermitianE}\label{subsec:lanczos}
When $A$ is Hermitian and $E = \eta \vy\vy^H$, $\eta \in \R$, the two families of linear systems in~\eqref{eq:shifted_systems_solutions} coincide, i.e., $\vx(t)=\vu(t)$, and it is reasonable to choose $\vxtilde(t)$ as \emph{Lanczos approximations}. First, an orthonormal basis $V_m = [\vv_1, \dots, \vv_m]$ of the Krylov subspace
\begin{equation}\label{eq:krylov_subspace}\nonumber
\spK_m(A,\vy) := \Span\{\vy,A\vy,A^2\vy,\dots,A^{m-1}\vy\}
\end{equation}
is computed via the short-recurrence Lanczos method~\cite{Lanczos1950}, collecting the orthonormalization coefficients in a tridiagonal, Hermitian matrix $T_m$.
The matrices $V_m$ and $T_m$ satisfy the \emph{Lanczos relation}
\begin{equation}\label{eq:lanczos_relation}
AV_m = V_mT_m + t_{m+1,m} \vv_{m+1}\ve_m^H,
\end{equation}
where $\ve_m$ denotes the $m$th canonical unit vector. Given these quantities, the Lanczos approximation is given as
\begin{equation}\label{eq:lanczos_approx_linearsystem}
\vx_m^\Lan(t) := \|\vy\|V_m (tI-T_m)^{-1}\ve_1.
\end{equation}
Substituting~\eqref{eq:lanczos_approx_linearsystem} into~\eqref{eq:frechet_derivative_integral_compact_approx} in place of $\vxtilde(t)$ gives the $m$th Lanczos approximation for the \Fd,
\begin{eqnarray}
L_m^\Lan &:=& \frac{\eta}{2\pi i} \int_\Gamma f(t) \|\vy\|^2 V_m (tI-T_m)^{-1}\ve_1\ve_1^H(tI-T_m)^{-1}V_m^H\d t \nonumber \\
&=& V_m \frac{\eta}{2\pi i} \int_\Gamma f(t) \|\vy\|^2(tI-T_m)^{-1}\ve_1\ve_1^H(tI-T_m)^{-1}\d t V_m^H\nonumber \\
&=& V_m L_f(T_m,\eta \|\vy\|^2 \ve_1\ve_1^H) V_m^H\label{eq:lanczos_approximation_frechet}.
\end{eqnarray}
Thus, computing the approximation~\eqref{eq:lanczos_approximation_frechet} amounts to computing the \Fd\ of the compressed matrix $T_m$ with respect to the direction term $\eta\|\vy\|^2\ve_1\ve_1^H$. As typically $m \ll n$, this can be done by standard methods for the \Fd\ of a small, dense matrix. In addition, it is directly obvious from the representation~\eqref{eq:lanczos_approximation_frechet} that the Lanczos approximation $L_m^\Lan$ is of rank at most $m$. We summarize the outlined approach in Algorithm~\ref{alg:lanczos_frechet}.

\begin{algorithm}
\caption{\label{alg:lanczos_frechet}Lanczos approximation of $L_f(A,\eta\vy\vy^H)$ for Hermitian $A$.}
\begin{algorithmic}[1]
\State \textbf{Input:} $m \in \N$, $A \in \Cnn$ Hermitian, $\eta \in \R$, $\vy \in \Cn$, function $f$
\State \textbf{Output:} Rank $m$ approximation $L_m^\Lan = \eta V_mX_mV_m^H \approx L_f(A,\eta\vy\vy^H)$
\State compute $V_m, T_m$ via $m$ Lanczos steps for $A$ and $\vy$\label{line:calltolanczos}
\State compute $X_m \leftarrow L_f(T_m,\eta\|\vy\|^2 \ve_1\ve_1^H)$
\If{desired}
\State form $L_m^\Lan \leftarrow \eta V_mX_mV_m^H$
\Else
\State return low-rank factors $V_m,X_m$
\EndIf
\end{algorithmic}
\end{algorithm}

\begin{remark}
It is often not necessary to form the approximation $L_m^\Lan$ explicitly, e.g., when only matrix-vector products with it need to be performed. In that case, storing the low-rank factors $V_m$ and $X_m$ requires memory of $\mathcal{O}(mn+m^2)$. If $m \ll n$ (as it will typically be the case in practice), this is significantly lower than $\mathcal{O}(n^2)$ needed for storing the full matrix $L_m^\Lan$.

Matrix-vector products with $L_m^\Lan$ can then efficiently be computed as
\begin{equation}\label{eq:low_rank_mult}
L_f(A,\eta\vy\vz^H)\vb \approx V_m(X_m(V_m^H(\eta \vb)))
\end{equation}
with computational complexity $\mathcal{O}(mn+m^2)$.

Let us brief\/ly compare this to the Krylov algorithm
from~\cite{KandolfRelton2016} for approximating $L_f(A,E)\vb$. In this
approach, the vector $\vb$ is part of the Krylov iteration so that the
method needs to be run again if a matrix-vector product with a vector
different from $\vb$ needs to be approximated.
In contrast,
by computing $V_m,X_m$ \emph{only once} with
Algorithm~\ref{alg:lanczos_frechet}
and then using~\eqref{eq:low_rank_mult},
we can efficiently approximate the action of $L_f(A,\eta\vy\vy^H)$ on any number of vectors.\hfill$\diamond$
\end{remark}

\begin{remark}
There are other possible motivations for arriving at the
approximation~\eqref{eq:lanczos_approximation_frechet}. One way 
is to consider the projection of the original problem of computing
$L_f(A,\eta\vy\vy^H)$ onto the tensorized Krylov subspace
$\spK_m(A,\vy) \otimes \spK_m(A,\vy)$, i.e.,
$$L_m^\Lan := V_m(L_f(V_m^HAV_m, V_m\eta\vy\vy^HV_m^H)V_m^H$$
which coincides with~\eqref{eq:lanczos_approximation_frechet} as
$V_m^HAV_m = T_m$ and $V_m\vy = \|\vy\|\ve_1$. This shows that our
approach is closely related to projection techniques for matrix
equations~\cite{Simoncini2016} or low-rank updates of matrix
functions~\cite{BeckermannKressnerSchweitzer2018}. We chose the above
approach based on the integral representation as this leads to a more
natural generalization to the non-Hermitian case, which will be
covered in the next subsection.

Another way of arriving at this approximation---which also handles the non-Hermitian case---is based on a general Krylov framework for bivariate matrix functions introduced by Kressner in~\cite{Kressner2019}. In particular, Algorithm~2 in~\cite[Section 5]{Kressner2019}, which was discovered independently from this work, coincides with our Algorithm~\ref{alg:arnoldi_frechet} discussed below.\hfill$\diamond$
\end{remark}

\subsection{Arnoldi approximation for the non-Hermitian case}\label{subsec:arnoldi}
In the non-Hermitian case $A \neq A^H$, no short-recurrence method for generating the Krylov basis vectors
$\vv_1, \dots,\vv_m$ exists in general. Instead, one can use the
\emph{Arnoldi method} which explicitly orthogonalizes $\vv_i$ against
all previous basis vectors $\vv_1,\dots,\vv_{i-1}$. In contrast to the
Hermitian case, the two linear
systems~\eqref{eq:shifted_systems_solutions} do not coincide, so that
two Krylov subspaces $\spK_m(A,\vy)$ and
$\spK_m(A^H\!,\vz)$ have to be built. Note that it is possible
to use different numbers $m_1 \neq m_2$ of steps for the two Krylov
subspaces, but for ease of presentation we always assume
$m_1 = m_2 = m$. Denoting the bases of $\spK_m(A,\vy)$ and
$\spK_m(A^H\!,\vz)$ by $V_m$ and $W_m$, respectively, and collecting
the corresponding Arnoldi orthonormalization coefficients in two upper
Hessenberg matrices $G_m$ and $H_m$, we obtain the \emph{Arnoldi
  relations}
\begin{eqnarray}
AV_m &=& V_mG_m + g_{m+1,m} \vv_{m+1}\ve_m^H \label{eq:arnoldi_relation1}\\
A^HW_m &=& W_mH_m + h_{m+1,m} \vw_{m+1}\ve_m^H \label{eq:arnoldi_relation2}.
\end{eqnarray}
The corresponding Arnoldi approximations for~\eqref{eq:shifted_systems_solutions} are then---analogously to~\eqref{eq:lanczos_approx_linearsystem}---given by
\begin{equation}\label{eq:arnoldi_approx_linearsystem}
\vx_m^\Arn(t) = \|\vy\|V_m (tI-G_m)^{-1}\ve_1 \quad \text{and} \quad\vu_m^\Arn(t) = \|\vz\|W_m (tI-H_m)^{-1}\ve_1.
\end{equation}
Plugging the approximations~\eqref{eq:arnoldi_approx_linearsystem} into~\eqref{eq:frechet_derivative_integral_compact_approx} gives the $m$th Arnoldi approximation for $L_f(A,\eta\vy\vz^H)$,
\begin{equation}\label{eq:arnoldi_approximation_frechet}
L_m^\Arn := V_m \frac{\eta}{2\pi i} \int_\Gamma f(t) \|\vy\|\|\vz\| (tI-G_m)^{-1}\ve_1\ve_1^H(tI-H_m^H)^{-1}\d t W_m^H.
\end{equation}
Here---in contrast to the Hermitian case---the obtained approximation
is not defined as the \Fd\ of a matrix of size $m
\times m$. Therefore it is at first sight not completely clear how to
evaluate~\eqref{eq:arnoldi_approximation_frechet} in an efficient
manner.
The following result, which was independently from this work also proven by Kressner in~\cite[Lemma~4]{Kressner2019}, allows us to evaluate the integral in~\eqref{eq:arnoldi_approximation_frechet} by computing a function of a $2m \times 2m$ block matrix; see also~\cite[Lemma 2.2]{BeckermannKressnerSchweitzer2018} for a similar result in the context of low-rank updates of matrix functions.

\begin{lemma}\label{lem:arnoldi_approx_block}
Let $f$ be of the form~\eqref{eq:cauchy_integral} and let $G_m, H_m$
from~\eqref{eq:arnoldi_relation1}--\eqref{eq:arnoldi_relation2} be such that $f(G_m), f(H_m^H)$ are defined. Let
\begin{equation}\label{eq:B}
B := \left[\begin{array}{cc}G_m & \eta\|\vy\|\|\vz\| \ve_1\ve_1^H\\ 0 & H_m^H\end{array}\right].
\end{equation}
Then
\begin{equation}\label{eq:fB}
f(B) = \left[\begin{array}{cc}f(G_m) & X_m \\ 0 & f(H_m^H)\end{array}\right] \text{ with } X_m = \frac{\eta}{2\pi i}\int_\Gamma f(t) \|\vy\|\|\vz\| (tI-G_m)^{-1}\ve_1\ve_1^H (tI-H_m^H)^{-1}\d t.
\end{equation}
\end{lemma}
\begin{proof}
For the inverse of $tI - B$, where $B$ is a a block matrix of the form~\eqref{eq:B}, block Gaussian elimination yields
$$(tI-B)^{-1} = \left[\begin{array}{cc}(tI-G_m)^{-1} & (tI-G_m)^{-1}E_m(tI-H_m^H)^{-1} \\ 0 & (tI-H_m^H)^{-1}\end{array}\right].$$
Integrating this componentwise for $t \in \Gamma$ gives the desired result.
\end{proof}

We have hence arrived at a rank $m$ approximation of the \Fd, which can be compactly written as
\begin{equation}\label{eq:arnoldi_approx_compact}
L_m^\Arn = \eta V_mX_mW_m^H.
\end{equation}
The resulting method is summarized in Algorithm~\ref{alg:arnoldi_frechet}. Note that if $A$ is Hermitian but $\vy \neq \vz$, we can use a variant of Algorithm~\ref{alg:arnoldi_frechet} in which the Arnoldi process in lines~\ref{line:calltoarnoldi1} and~\ref{line:calltoarnoldi2} is replaced by the Lanczos process.

\begin{algorithm}
\caption{\label{alg:arnoldi_frechet}Arnoldi approximation of $L_f(A,\eta\vy\vz^H)$.}
\begin{algorithmic}[1]
\State \textbf{Input:} $m \in \N$, $A \in \Cnn$, $\eta \in \C$, $\vy, \vz \in \Cn$, function $f$
\State \textbf{Output:} Rank $m$ approximation $L_m^\Arn = \eta V_mX_mW_m^H \approx L_f(A,\eta\vy\vz^H)$
\State compute $V_m, G_m$ via $m$ Arnoldi steps for $A$ and $\vy$\label{line:calltoarnoldi1}
\State compute $W_m, H_m$ via $m$ Arnoldi steps for $A^H$ and $\vz$\label{line:calltoarnoldi2}
\State compute $X_m$ via~\eqref{eq:fB}
\If{desired}
\State form $L_m^\Arn \leftarrow \eta V_mX_mW_m^H$
\Else
\State {return low-rank factors $V_m,W_m,X_m$}
\EndIf
\end{algorithmic}
\end{algorithm}

\begin{remark}
We brief\/ly remark that the result of Lemma~\ref{lem:arnoldi_approx_block} can be seen as a generalization of the formula~\eqref{eq:2x2block} which relates the \Fd\ to the (1,2)-block of $f$ evaluated on a block matrix. In particular, when $A = A^H$, $\vy = \vz$, we have $V_m = W_m$ and both matrices $G_m$ and $H_m^H$ coincide with the tridiagonal matrix $T_m$ from the Lanczos process. Thus, the matrix~\eqref{eq:B} has the form
$$\left[\begin{array}{cc}T_m & \eta\|\vy\|^2 \ve_1\ve_1^H\\ 0 & T_m\end{array}\right],$$
so that by~\eqref{eq:2x2block}, we find
$$X_m = L_f(T_m,\eta\|\vy\|^2\ve_1\ve_1^H),$$
i.e., the approximation~\eqref{eq:arnoldi_approx_compact} agrees with~\eqref{eq:lanczos_approximation_frechet} in the Hermitian case.\hfill$\diamond$
\end{remark}

\subsection{Two-sided Lanczos for non-Hermitian $A$}\label{subsec:twosided}
An alternative to using the Arnoldi method when $A \neq A^H$, $\vy \neq \vz$ is to use the two-sided Lanczos method~\cite[Section 7.1]{Saad2003} (sometimes also called non-Hermitian Lanczos). If $\vy^H\!\!\vz = 1$ (which we can always assume without loss of generality as long as $\vy$ and $\vz$ are not orthogonal to each other), this method uses a coupled three-term recursion to compute \emph{bi-orthonormal} bases $\Vhat_m, \What_m$ of $\spK_m(A,\vy)$ and $\spK_m(A^H\!,\vz)$, respectively, i.e., $\Vhat_m^H\What_m = I$. Note that one iteration of the two-sided Lanczos process requires performing \emph{two} matrix vector products, one with $A$ and one with $A^H$, so that the number of matrix vector products for computing the bi-orthonormal bases $\Vhat_m$ and $\What_m$ in the two-sided Lanczos method is the same as that of computing the orthonormal bases $V_m, W_m$ in the Arnoldi method outlined in section~\ref{subsec:arnoldi}.

We denote by
\begin{equation}\label{eq:WAV}\nonumber
\That_m = \What_m^H\!\!A \Vhat_m,
\end{equation}
the orthogonal projection of $A$ onto $\spK_m(A,\vy)$ along $\spK_m(A^H\!,\vz)$, which is tridiagonal and contains the coefficients from the bi-orthonormalization procedure. Then, we have the following \emph{two-sided Lanczos relations}
\begin{eqnarray}
A\Vhat_m &=& \Vhat_m\That_m + \widehat{t}_{m+1,m}\vvhat_{m+1}\ve_m^H,\label{eq:twosided_lanczos_relation1}\nonumber\\
A^H\What_m &=& \What_m\That_m^H + \widehat{t}_{m,m+1}\vwhat_{m+1}\ve_m^H.\label{eq:twosided_lanczos_relation2}\nonumber
\end{eqnarray}
The corresponding approximations for the solutions of the shifted linear systems~\eqref{eq:shifted_systems_solutions} are then given by
\begin{equation}\label{eq:linear_systems_approximations_twosided}
\widehat{\vx}_m(t) := \|\vy\|\Vhat_m(tI-\That_m)^{-1}\ve_1 \quad\text{ and }\quad \widehat{\vu}_m(t) := \|\vz\| \What_m(tI-\That_m)^{-H}\ve_1.
\end{equation}
As before, we substitute the approximations~\eqref{eq:linear_systems_approximations_twosided} into~\eqref{eq:frechet_derivative_integral_compact_approx} and obtain the approximation
\begin{equation}
\widehat{L}_m = \Vhat_m L_f(\That_m,\eta \|\vy\|  \|\vz\| \ve_1\ve_1^H) \What^H,\label{eq:twosided_approximation_frechet}\nonumber
\end{equation}
for $L_f(A,\eta\vy\vz^H)$.
We will refrain from giving an explicit algorithm for this approach,
as it is a completely straight-forward modification of Algorithm~\ref{alg:lanczos_frechet}.

\begin{remark}\label{rem:breakdown}
A potential disadvantage of the two-sided Lanczos method when compared to the standard Arnoldi and Lanczos method is the possibility of a \emph{serious breakdown}. This happens when $\widehat{\vv}_{j}^H\widehat{\vw}_{j} = 0$ with $\widehat{\vv}_{j} \neq \vnull$ and $\widehat{\vw}_{j} \neq \vnull$ for some $j$. In that case, the two-sided Lanczos iteration in its most basic form cannot be continued. We will, however, not go into detail on this topic, as serious breakdowns very rarely appear in practice, and standard \emph{look-ahead} techniques for avoiding breakdowns, as discussed in, e.g.,~\cite{FreundGutknechtNachtigal1993,ParlettTaylorLiu1985}, can be straightforwardly used in our setting. While there are also so-called \emph{incurable} breakdowns that cannot be prevented by look-ahead techniques~\cite{ParlettTaylorLiu1985}, these techniques typically work well in practical situations.
\end{remark}

\subsection{Block Lanczos for Hermitian $A$}\label{subsec:block}

When $A$ is Hermitian, but $\vy \neq \vz$, Algorithm~\ref{alg:lanczos_frechet} cannot be used. Instead of using the Arnoldi-based Algorithm~\ref{alg:arnoldi_frechet}, it is also possible to use a \emph{block Lanczos approach}, see, e.g.,~\cite{Saad2003}, or~\cite{FrommerLundSzyld2018,Lund2018} for recent work concerning the usage of block Krylov subspace methods in the matrix function context.

Collecting the two vectors $\vy,\vz$ in a block vector $Y = [\vy, \vz] \in \C^{n \times 2}$, we construct an orthonormal basis $V_m^\Box$ of the \emph{block Krylov subspace}
\begin{equation}\label{eq:block_krylov_subspace}
\spK_m^\Box(A,Y) := \Span\{Y, AY, A^2Y,\dots,A^{m-1}Y\}.
\end{equation}
We give one possible basic implementation of the block Lanczos method as Algorithm~\ref{alg:block_lanczos}. For ease of presentation, we assume that all $V_j$ computed throughout the algorithm are linearly independent, i.e, that the block Krylov subspace~\eqref{eq:block_krylov_subspace} is of full dimension $2m$. If this is not the case, special care has to be taken in order to remove linearly dependent vectors, a process known as \emph{deflation}, see, e.g.,~\cite{Ruhe1979}.

\begin{algorithm}
\caption{\label{alg:block_lanczos}Block Lanczos process for a Hermitian $A$.}
\begin{algorithmic}[1]
\State \textbf{Input:} $m \in \N$, $A \in \Cnn$ Hermitian, $Y \in \C^{n \times 2}$
\State \textbf{Output:} Orthonormal base $\Vcal_m^\Box = [V_1,\dots,V_{m}]$ of $\spK^\Box_m(A,V)$,
block tridiagonal matrix $\Tcal^\Box_m = (\Vcal^\Box)^H_m A \Vcal^\Box_m$
\For{$j=1,2,\dots$,m}
\State $T_{ij} \leftarrow V_i^* A V_j, \quad i = 1, \ldots, j$ 
\State $i_0 \leftarrow \max\{1,j-2\}$ 
\State $W_j \leftarrow AV_j - \sum_{i=i_0}^j V_i T_{ij}$ 
\State Compute QR decomposition $W_j = V_{j+1} T_{j+1,j}$ 
\EndFor
\State $\Tcal^\Box_m \leftarrow (\Vcal_m^\Box)^H A \Vcal^\Box_m = (T_{ij})_{i,j=1}^m$ 
\State $\Vcal^\Box_m \leftarrow [V_1,\dots,V_{m}]$ 
\end{algorithmic}
\end{algorithm}

The block tridiagonal matrix $\Tcal^\Box_m$ of (block-)orthogonalization coefficients satisfies a block analogue of~\eqref{eq:lanczos_relation}
\begin{equation}\label{eq:block_lanczos_relation}\nonumber
A\Vcal_m^\Box = \Vcal_m^\Box\Tcal_m^\Box + V_{m+1}T_{m+1,m}E_m^H
\end{equation}
where $T_{m+1,m} \in \C^{2 \times 2}$ and $E_m = [\ve_{2m-1}, \ve_{2m}]$. In particular $\Tcal^\Box_m = (\Vcal_m^\Box)^HA\Vcal_m^\Box$. As
$$\spK_m^\Box(A,Y) = \spK_m(A,\vy) \cup \spK_m(A,\vz),$$
approximations to both $\vx(t)$ and $\vu(t)$ from~\eqref{eq:shifted_systems_solutions} can be extracted from $\spK_m^\Box(A,Y)$. The standard choice for these approximations is given by
\begin{equation}\label{eq:shifted_linear_systems_approximations2}\nonumber
\begin{aligned}
\vx_m^\Box(t) &:=  \Vcal_m^\Box (tI-\Tcal_m^\Box)^{-1} (\Vcal_m^\Box)^H \vy,\\
\vu_m^\Box(t) &:=  \Vcal_m^\Box (tI-\Tcal_m^\Box)^{-1} (\Vcal_m^\Box)^H \vz.\\
\end{aligned}
\end{equation}
An approximation to the \Fd\ is then obtained in the usual way as
\begin{equation} \label{eq:block_lanczos_approx}\nonumber
L_m^\Box =  \Vcal^\Box_m L_f(\Tcal^\Box_m, \eta \vy_m \vz_m^H) \Vcal_m^H,
\end{equation}
where $\vy_m = (\Vcal_m^\Box)^H \vy$, $\vz_m = (\Vcal_m^\Box)^H \vz$.
We again refrain from giving an explicit algorithm for this approach.

\begin{remark}
There are two main advantages of using a block Krylov approach over the standard Krylov approach from section~\ref{subsec:arnoldi}: The standard Arnoldi method requires $2m$ matrix vector products with $A$, while the block Lanczos algorithm requires $m$ matrix block vector products with blocks of size $n \times 2$. While mathematically, this amounts to the same number of operations, one can typically implement matrix block vector products such that they benefit from more cache-friendly memory access and require \emph{less} computation time then an equivalent number of individual matrix vector products, see, e.g.,~\cite{BakerDennisJessup2006}.

In addition, the block Krylov space $\spK^\Box_m(A,Y)$ is the union of the Krylov subspaces for $\vy$ and $\vz$, the approximate solution for one system can also use information contained in the Krylov subspace for the \emph{other} system. Therefore, a smaller overall subspace dimension may potentially suffice to reach the desired accuracy. In~\cite{Birk2015}, savings of up to 35\% in the number of iterations are reported (for block sizes larger than two). This largely depends on the vectors $\vy, \vz$ though and is difficult to quantify in advance.

On the other hand, block Krylov methods are much more complicated to implement efficiently, especially if one wants to take proper care of issues like deflation.\hfill$\diamond$
\end{remark}

\subsection{Extended and rational Krylov subspace methods}\label{subsec:rational}
All approaches outlined so far have been based on standard
(polynomial) Krylov subspaces. Of course, we can also use other
projection spaces in our methods. In recent years, \emph{rational
  Krylov subspaces} have been successfully applied in matrix function
computations
(see e.g.,~\cite{DruskinKnizhnerman1998,Guettel2013,GuettelKnizhnerman2013,KnizhnermanSimoncini2010})
and often have much better approximation properties than polynomial
Krylov spaces. Therefore, it is natural to also consider these
subspaces in our projection approach. Rational Krylov subspaces are of
the form
\begin{equation}\label{eq:rational_ks}\nonumber
\spQ_m(A,\vy) = q_{m-1}(A)^{-1}\spK_m(A,\vy),
\end{equation}
where $q_{m-1}(z) = (z-\xi_1)(z-\xi_2)\cdots(z-\xi_{m-1})$ is a polynomial of degree $m-1$. The scalars $\xi_1,\dots,\xi_{m-1} \in \C \cup \{\infty\}$ are called the \emph{poles} of the rational Krylov subspace. Similarly to the polynomial case we can define an approximation for $L_f(A,\eta\vy\vz^H)$ based on rational Krylov subspaces. We just brief\/ly summarize the approach for the Hermitian case, the extension to the non-Hermitian case is then straightforward.

For this, let $\tildeV$ denote an orthonormal basis of $\spQ_m(A,\vy)$,
which can be computed by the \emph{rational Arnoldi method}
(see e.g.,~\cite{Ruhe1984,Ruhe1994}) and let
$\tildeT = \tildeV^H A \tildeV$ denote the compression of $A$ onto
$\spQ_m(A,\vy)$. Then, an approximation for $L_f(A,\eta \vy\vy^H)$ is
obtained completely analogously to the polynomial case as
\begin{equation}\label{eq:rational_approximation_frechet}
L_m^\Rat := \tildeV L_f(\tildeT,\eta \|\vy\|^2 \ve_1\ve_1^H) \tildeV^H.
\end{equation}

The rational Arnoldi algorithm requires (for finite poles) the
solution of a (shifted) linear system with $A$ in each iteration, in
addition to a matrix-vector product. Also note that, even in the
Hermitian case, no short recurrences for the basis vectors exist in
general. An exception to this are \emph{extended Krylov subspaces} which only use the poles
$0$ and $\infty$ (see e.g.,~\cite{DruskinKnizhnerman1998,Simoncini2007,JagelsReichel2011}).

The efficiency of using rational Krylov subspace methods thus largely depends on how efficiently shifted systems with $A$ can be solved, and how often the poles vary---when using a direct solver, one Cholesky factorization needs to be computed per pole. In cases where $A$ is banded with rather small bandwidth, rational Krylov methods are thus particularly attractive. An additional benefit of the lower iteration number when using a rational Krylov method in our setting is that it also implies that the resulting approximation $L_m^\Rat$ is of lower rank than when using a polynomial method, such that it requires less storage and subsequent matrix-vector products with it are less costly.

\begin{remark}
The usage of rational Krylov subspaces can of course be combined with a block Krylov approach similar to that of section~\ref{subsec:block}, leading to a \emph{rational block Krylov method}, see, e.g.~\cite{ElsworthGuettel2020}. A combination of rational Krylov subspaces with a two-sided approach as in section~\ref{subsec:twosided} is in principle also possible, but as there are no short recurrences even in the Hermitian case, there also do not exist short-recurrence two-sided rational methods for the non-Hermitian case. The approach of using bi-orthonormal bases thus does not seem very attractive in this setting. An exception is the extended Krylov case, for which a two-sided short-recurrence method was recently derived by Schweitzer in~\cite{Schweitzer2017}.\hfill$\diamond$
\end{remark}

\section{Convergence analysis for Hermitian $A$}\label{sec:convergence}

In this section, we investigate the convergence behavior of the proposed Krylov subspace methods for approximating the \Fd. We restrict ourselves to the case of Hermitian $A$ and standard polynomial Krylov methods. An extension of the result to block Krylov methods is possible in a straight-forward way. Let us note that Kressner also provides a convergence result for Krylov approximations to the \Fd{} in~\cite[Corollary~1]{Kressner2019}, which relates the error of the Krylov approximation to the error of a polynomial approximation of $f'$; see also recent work by Crouzeix and Kressner~\cite[Corollary~6.1]{CrouzeixKressner2020}.

We begin by stating a result for the exponential function of a Hermitian negative semidefinite matrix. The technique of proof used for this result largely resembles that of the famous convergence result of Hochbruck and Lubich for $\exp(A)\vb$, see~\cite{HochbruckLubich1997}. We state the result for the approximation $L_m^\Arn$ from~\eqref{eq:arnoldi_approximation_frechet} in order to cover the more general case $\vy \neq \vz$. Of course, it holds in the same way for $L_m^\Lan$ from~\eqref{eq:lanczos_approximation_frechet} when $\vy = \vz$.

\begin{theorem} \label{thm:apriori} \em
Suppose $A$ is Hermitian negative semidefinite with its spectrum inside the interval $[-4\rho, 0]$. Then, we have for the error
$\varepsilon_m := \|L_f(A, \vy \vz^H) - L_m^\Arn\|$ the bound
\begin{equation} \label{eq:apriori_bounds}
\begin{aligned}
\varepsilon_m & \leq  10 \, \frac{(4\rho \tau)^2}{m^2} \, e^{-m^2/(5\rho \, t)} \| \vy \| \| \vz \|,
\quad \quad \quad \sqrt{4 \rho \, t} \leq m \leq 2 \rho \, t, \\
\varepsilon_m & \leq \frac{40}{\rho \, t} \, e^{- \rho \, t} \left( \frac{e  \rho \, t}{m} \right)^m \| \vy \| \| \vz \|,
\quad \quad \quad  \quad m \geq 2 \rho t.
\end{aligned}
\end{equation}
\end{theorem}
\begin{proof}
Recall first the notation
\begin{equation*}
\vx_m(t) = \|\vy\|V_m(tI-T_m)^{-1}\ve_1 \quad\text{ and }\quad \vu_m(t) = \|\vz\| W_m(tI-T_m)^{-H}\ve_1.
\end{equation*}
By adding and subtracting $(tI - A)^{-1} \vy \, \vu_m(t)$ in the integrand, we see that
\begin{equation} \label{eq:eq_1}
\begin{aligned}
 L_f(A, y z^H) - L_m^\Arn
=  \frac{1 }{2 \pi i } \int_\Gamma f(t) \left[
( \Delta_m(t) \vy  \vu_m(t) )  + (t I - A)^{-1} \vy \vz^H \Delta_m(t) \right]\d t,
\end{aligned}
\end{equation}
where
\begin{equation} \label{eq:eq_2}\nonumber
\begin{aligned}
\Delta_m(t) =  (tI - A)^{-1} -  V_m (tI- T_m)^{-1} V_m^H.
\end{aligned}
\end{equation}
Then, using the bounds
\begin{equation} \label{eq:resolvent_bounds_1}
\| (tI-A)^{-1} \| \leq \frac{1}{d(z, \mathcal{F}(A))} \quad \textrm{and} \quad \| V_m(tI-T_m)^{-1}V_m^H \| \leq \frac{1}{d(z, \mathcal{F}(A))},
\end{equation}
%
the problem of bounding
\begin{equation} \label{eq:eq_3}\nonumber
\| \Delta_m(t) \vy \|, \quad \textrm{and} \quad \| \Delta_m(t) \vz \|,
\end{equation}
can be turned
into a polynomial approximation problem on the complex plane, as in \cite{BR09} and \cite{HochbruckLubich1997}.
To obtain the bounds \eqref{eq:apriori_bounds} we inspect Lemma 1 and Theorem 2 of~\cite{HochbruckLubich1997}.
From \eqref{eq:resolvent_bounds_1} it clearly follows that
\begin{equation} \label{eq:resolvent_bounds_2}
\|\vu_m(t)\| \leq \frac{\| \vz \|}{d(z, \mathcal{F}(A))} \quad
\textrm{and} \quad \|(t I - A)^{-1}\vy\| \leq \frac{\|\vy\|}{d(z, \mathcal{F}(A))}.
\end{equation}
Take $\bbE$ to be a convex set in the complex plane satisfying the
conditions of~\cite[Lemma\;1]{HochbruckLubich1997}.
From \eqref{eq:eq_1}, \eqref{eq:resolvent_bounds_2} and Lemma 1 of~\cite{HochbruckLubich1997},
we see that for the norm of the first term of \eqref{eq:eq_1}, i.e., for
$$
\varepsilon_{1,m}(t) :=  \left\| \frac{1 }{2 \pi i } \int_\Gamma f(t)   \Delta_m(t) \vy  \vu_m(t)   \, \d t \right\|,
$$
the bound Lemma 1 of~\cite{HochbruckLubich1997} holds with the constant $M$ replaced by
$M = \ell(\partial \bbE)/ [d(\partial \bbE) \cdot d(\Gamma)^2 ]$ (and multiplied by $\|\vz\| \| \vy \|$).
Then, choosing the contour $\Gamma$ as in proof of~\cite[Theorem 2]{HochbruckLubich1997},
we see that instead of the bound (3.4) of~\cite[Theorem 2]{HochbruckLubich1997},
we have
$$
\varepsilon_{1,m}(t) \leq
\frac{ e^{ 2 \rho \tau \epsilon} r^{-m}}{\epsilon}
\left( \frac{1 + \epsilon}{\rho \tau \epsilon}
+ \sqrt{\frac{(2+\epsilon)\pi}{\rho \tau \epsilon} }\right)  \| \vz \| \| \vy \|,
$$
where $\epsilon$ can be chosen freely. We choose as in~\cite[Theorem 2]{HochbruckLubich1997}
$$
\epsilon = \frac{m^2}{8 (\rho \tau)^2}.
$$
For $\epsilon \leq \frac{1}{2}$, i.e., for $m \leq 2 \rho \tau$, we have the bound (3.1) of~\cite[Theorem 2]{HochbruckLubich1997}
multiplied by $\epsilon^{-1}$. When $m\geq 2\rho \tau$, $\epsilon^{-1} \leq 2$, and
we have the bound (3.2) multiplied by 2.
The second term of \eqref{eq:eq_1} can be bounded similarly, from which the bound \eqref{eq:apriori_bounds} follows.
\end{proof}

\begin{example}\label{eg:Thm31}
 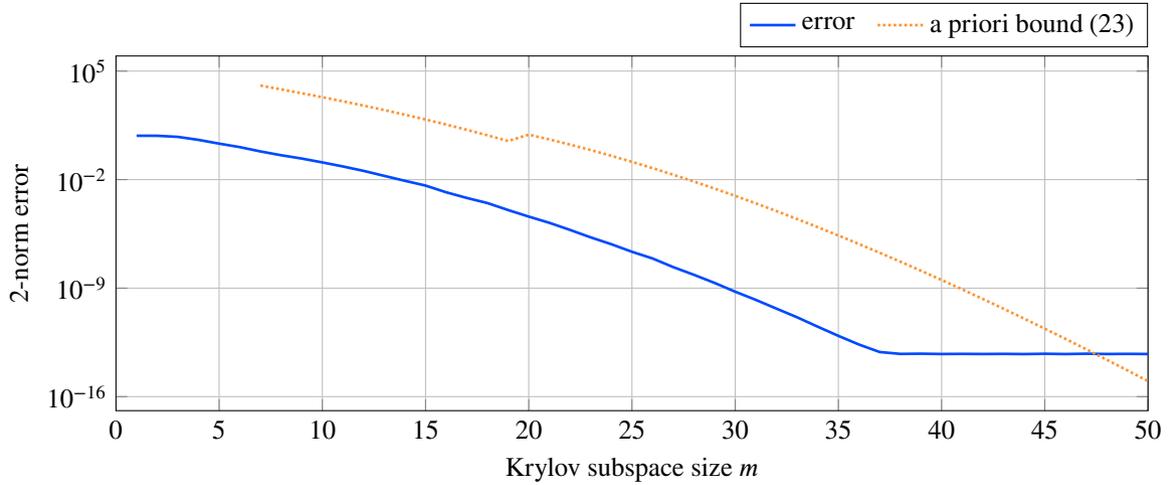
\begin{figure}
\centering
\tikzsetnextfilename{apriori}
\pgfplotsset{height=0.35\linewidth,width=0.85\linewidth,compat=1.10,every axis/.append style={legend style={/tikz/every even column/.append style={column sep=6pt}}}}
\pgfplotscreateplotcyclelist{list_apri}{%
color_peter1,line width=1pt, mark=none, solid\\
color_peter3,line width=1pt, mark=none, densely dotted\\
}
\noindent%
\begin{tikzpicture}[scale=1]%
    \begin{semilogyaxis}[legend style={at={(1,1.15)}}, 
   	anchor= north east, legend columns=5,cycle list name=list_apri, 
   	xmin=0, xmax=50,grid=major, 
   	xlabel={Krylov subspace size $m$}, ylabel={2-norm error}]
		\addplot+[] 
			table[header=false, x expr=\lineno+1, y index={0}, row sep=\\]
		   {6.7821e+00\\ 6.7277e+00\\ 5.7270e+00\\ 3.6629e+00\\ 2.1015e+00\\ 
		    1.2518e+00\\ 6.5868e-01\\ 3.7494e-01\\ 2.2887e-01\\ 1.2886e-01\\ 
		    7.0517e-02\\ 3.6572e-02\\ 1.7465e-02\\ 8.3720e-03\\ 4.1346e-03\\ 
		    1.5258e-03\\ 6.5697e-04\\ 3.0702e-04\\ 1.0899e-04\\ 4.1148e-05\\ 
		    1.6324e-05\\ 5.6845e-06\\ 1.8833e-06\\ 6.8186e-07\\ 2.2136e-07\\ 
		    8.1721e-08\\ 2.2681e-08\\ 7.2674e-09\\ 2.1698e-09\\ 5.8976e-10\\ 
		    1.7450e-10\\ 4.7601e-11\\ 1.3147e-11\\ 3.2536e-12\\ 8.3400e-13\\ 
		    2.2923e-13\\ 7.5671e-14\\ 5.7110e-14\\ 5.8910e-14\\ 5.6268e-14\\ 
		    5.7336e-14\\ 5.6252e-14\\ 5.7326e-14\\ 5.5642e-14\\ 5.8316e-14\\ 
		    5.5797e-14\\ 5.8773e-14\\ 5.6271e-14\\ 5.7311e-14\\ 5.6278e-14\\};
	\addlegendentry{error}
		\addplot+[] 
			table[header=false, x expr=\lineno+1, y index={0}, row sep=\\]
		   {0.0000e+00\\ 0.0000e+00\\ 0.0000e+00\\ 0.0000e+00\\ 0.0000e+00\\ 
		   	0.0000e+00\\ 1.1704e+04\\ 6.6378e+03\\ 3.7327e+03\\ 2.0675e+03\\ 
		   	1.1226e+03\\ 5.9539e+02\\ 3.0767e+02\\ 1.5457e+02\\ 7.5380e+01\\ 
		   	3.5635e+01\\ 1.6312e+01\\ 7.2241e+00\\ 3.0929e+00\\ 8.0129e+00\\ 
		   	3.9082e+00\\ 1.8175e+00\\ 8.0764e-01\\ 3.4362e-01\\ 1.4023e-01\\ 
		   	5.4982e-02\\ 2.0744e-02\\ 7.5419e-03\\ 2.6458e-03\\ 8.9669e-04\\ 
		   	2.9394e-04\\ 9.3296e-05\\ 2.8701e-05\\ 8.5656e-06\\ 2.4823e-06\\ 
		   	6.9908e-07\\ 1.9149e-07\\ 5.1053e-08\\ 1.3258e-08\\ 3.3556e-09\\ 
		   	8.2838e-10\\ 1.9957e-10\\ 4.6946e-11\\ 1.0790e-11\\ 2.4242e-12\\ 
		   	5.3266e-13\\ 1.1453e-13\\ 2.4105e-14\\ 4.9690e-15\\ 1.0036e-15\\ };
	\addlegendentry{a priori bound \eqref{eq:apriori_bounds}}
    \end{semilogyaxis}
\end{tikzpicture}
 \caption{Convergence v.s.\ a priori bound \eqref{eq:apriori_bounds}, for Example~\ref{eg:Thm31}.}
 \label{fig:apriori}
 \end{figure}
Consider the following simple numerical example to illustrate the
bound given by Theorem~\ref{thm:apriori}.
Set $A = 10 \cdot \mathrm{diag}(1,\ -2,\ 1) \in \mathbb{R}^{n \times n}$, and take randomly $\vy \in \mathbb{R}^n$
and $\vz \in \mathbb{R}^n$. Set $n=100$. Figure~\ref{fig:apriori} shows the convergence of the approximation \eqref{eq:arnoldi_approximation_frechet} v.s.\ the bound given by Theorem~\ref{thm:apriori}.\hfill$\diamond$
\end{example}

Next, we prove a result for the class of Stieltjes functions. It is based on the classical convergence result for the conjugate gradient method (CG)~\cite{HestenesStiefel1952}. It bounds the energy norm
\begin{equation*}
\|\ve\|_A = \sqrt{\ve^H \!A \ve}
\end{equation*}
of the error, and we restate it here for the sake of completeness.

\begin{theorem}[see, e.g.,~\cite{Saad2003}]\label{the:cg_convergence} \em
Let $A \in \Cnn$ be Hermitian positive definite and $\vx_0,\vy \in \Cn$. Further, let $\vx^\ast$ denote the exact solution of the linear system $A\vx = \vy$, let $\vx_m$ be the $m$th CG iterate with initial guess $\vx_0$ and let $\kappa$ denote the Euclidean norm condition number of $A$. Then the error in the CG method satisfies
\begin{equation*}
\|\vx^\ast - \vx_m\|_A \leq 2\left(\frac{\sqrt{\kappa}-1}{\sqrt{\kappa}+1}\right)^m\|\vx^\ast-\vx_0\|_A.
\end{equation*}
\end{theorem}

The proof of the following result, based on
Theorem~\ref{the:cg_convergence}, can be seen as a combination of
ideas used in~\cite{FrommerGuettelSchweitzer2014} for proving
convergence of restarted Krylov subspace methods for approximating
Stieltjes matrix functions and techniques used for analyzing
convergence of Krylov subspace methods for Lyapunov matrix equations
in~\cite{SimonciniDruskin2009}. We again use the more general case of the approximation $L_m^\Arn$ from~\eqref{eq:arnoldi_approximation_frechet}.

\begin{theorem}\label{the:convergence_stieltjes} \em
Let $A \in \Cnn$ be Hermitian positive definite, let $\vy,\vz \in \Cn$ with $\|\vy\| = \|\vz\| = 1$ and let $f$ be a Stieltjes function~\eqref{eq:stieltjes_integral}. Then the iterates $L_m^\Arn$ of Algorithm~\ref{alg:arnoldi_frechet} satisfy
\begin{equation}\label{eq:convergence_bound_stieltjes}
\|L_f(A,\eta\vy\vz^H) - L_m^\Arn\| \leq 4 |\eta \, f^\prime(\lmin)| \left(\frac{\sqrt{\kappa}-1}{\sqrt{\kappa}+1}\right)^m,
\end{equation}
where $\lmin$ is the smallest eigenvalue of $A$ and $\kappa$ denotes the Euclidean norm condition number of $A$.
\end{theorem}
\begin{proof}
Subtracting the integral representations~\eqref{eq:frechet_derivative_integral_compact} and~\eqref{eq:arnoldi_approximation_frechet}---modified to account for the fact that $f$ is a Stieltjes function---gives
\begin{eqnarray*}
L_f(A,\eta\vy\vz^H)-L_m^\Arn &=& \eta\int_0^\infty \vx(t)\vu(t)^H - \vx_m(t)\vu_m(t)^H \dmu\\
&=& \eta\int_0^\infty \vx(t)\big(\vu(t)-\vu_m(t)\big)^H \!\!+ \big(\vx(t)-\vx_m(t)\big)\vu_m(t)^H \dmu
\end{eqnarray*}
Taking the Euclidean norm on both sides then allows to estimate
\begin{equation}\label{eq:stieltjes_convergence_proof1}
\|L_f(A,\eta\vy\vz^H)-L_m^\Arn\| \leq |\eta|\int_0^\infty \|\vx(t)\|\|\vu(t)-\vu_m(t)\| + \|\vu_m(t)\|\|\vx(t)-\vx_m(t)\| \dmu.
\end{equation}
We proceed by looking at the integrand
\begin{equation}\label{eq:stieltjes_convergence_proof2}\nonumber
\|\vx(t)\|\|\vu(t)-\vu_m(t)\| + \|\vu_m(t)\|\|\vx(t)-\vx_m(t)\|,
\end{equation}
for fixed $t \geq 0$. In order to be able to use the conjugate gradient convergence result from Theorem~\ref{the:cg_convergence} to bound the right-hand side of~\eqref{eq:stieltjes_convergence_proof1}, we bound the Euclidean norm by the energy norm induced by the shifted matrix $A+tI$, using the relation
\begin{equation}\label{eq:norm_equivalence}\nonumber
\|\vv\| \leq \frac{1}{\sqrt{\lmin+t}}\|\vv\|_{A+tI}.
\end{equation}
From this, we obtain
$$\|\vx(t)-\vx_m(t)\| \leq \frac{2}{\sqrt{\lmin+t}}\left(\frac{\sqrt{\kappa(t)}-1}{\sqrt{\kappa(t)}+1}\right)^m\|\vx(t) - \vx_0(t)\|_{A+tI}$$
where $\kappa(t)$ denotes the Euclidean norm condition number of $A+tI$. As the iterates $\vx_m(t)$ correspond to choosing an initial guess $\vx_0(t) = \vnull$, we have
$$\|\vx(t) - \vx_0(t)\|_{A+tI} = \|\vx(t)\|_{A+tI} \leq \frac{1}{\sqrt{\lmin+t}},$$
and the same estimates can obviously be performed for the term $\|\vu(t)-\vu_0(t)\|$.

Further estimating $\|\vx(t)\| \leq \frac{1}{\lmin+t}$ and $\|\vu_m(t)\| \leq \frac{1}{\lmin+t}$, we obtain
\begin{equation}\label{eq:stieltjes_convergence_proof3}
\|\vx(t)\|\|\vu(t)-\vu_m(t)\| + \|\vu_m(t)\|\|\vx(t)-\vx_m(t)\| \leq \frac{4}{(\lmin+t)^2}\left(\frac{\sqrt{\kappa(t)}-1}{\sqrt{\kappa(t)}+1}\right)^m.
\end{equation}
Inserting~\eqref{eq:stieltjes_convergence_proof3} into~\eqref{eq:stieltjes_convergence_proof1}, we find
\begin{equation}\label{eq:stieltjes_convergence_proof4}\nonumber
\|L_f(A,\eta\vy\vz^H)-L_m^\Arn\| \leq |\eta|\int_0^\infty  \frac{4}{(\lmin+t)^2}\left(\frac{\sqrt{\kappa(t)}-1}{\sqrt{\kappa(t)}+1}\right)^m \dmu.
\end{equation}
Using the fact that $(\sqrt{\kappa(t)}-1)/(\sqrt{\kappa(t)}+1)$ is monotonically decreasing in $t$ and noting that
\begin{equation}\label{eq:stieltjes_derivative}\nonumber
f^\prime(z) = -\int_0^\infty \frac{1}{(t+z)^2} \dmu,
\end{equation}
see, e.g.,~\cite{AlzerBerg2002}, we obtain the desired result.
\end{proof}

\begin{remark}\label{rem:relation_to_kressner}
Let us note that it would be possible to obtain a result similar to that of Theorem~\ref{thm:apriori} from the work by Kressner~\cite{Kressner2019} and Crouzeix and Kressner~\cite{CrouzeixKressner2020}, but that this is indeed not easily possible for the case of Stieltjes functions treated in Theorem~\ref{the:convergence_stieltjes} as this would require a polynomial approximation result for derivatives of Stieltjes functions, which to our knowledge is not readily available in the literature.\hfill$\diamond$
\end{remark}

\begin{example}\label{ex:invsqrt_apriori}
 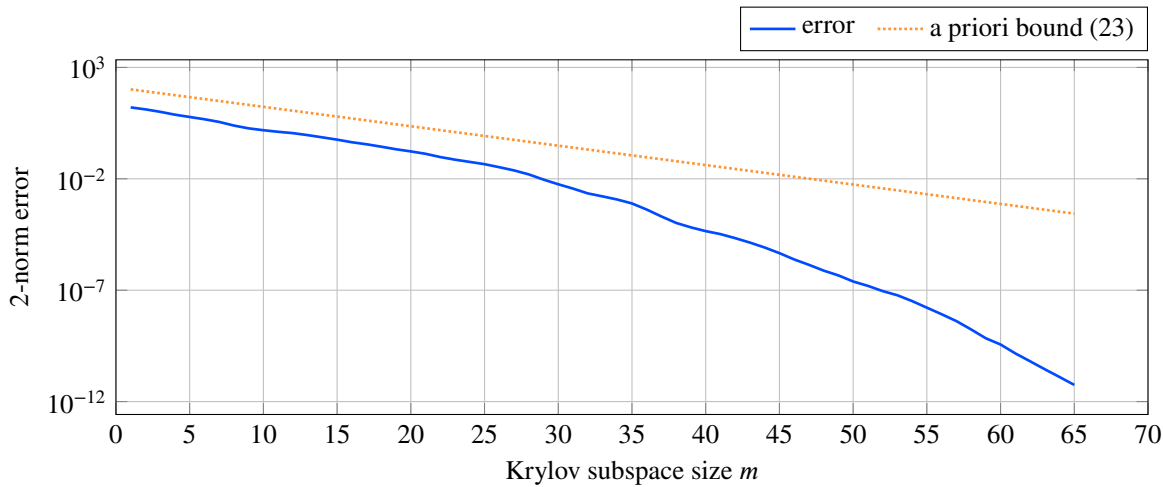
\begin{figure}
 \centering
\tikzsetnextfilename{apriori_invsqrt}
\pgfplotsset{height=0.35\linewidth,width=0.85\linewidth,compat=1.10,every axis/.append style={legend style={/tikz/every even column/.append style={column sep=6pt}}}}
\pgfplotscreateplotcyclelist{list_apri}{%
color_peter1,line width=1pt, mark=none, solid\\
color_peter3,line width=1pt, mark=none, densely dotted\\
}
\noindent%
\begin{tikzpicture}[scale=1]%
    \begin{semilogyaxis}[legend style={at={(1,1.15)}}, 
   	anchor= north east, legend columns=5,cycle list name=list_apri, 
   	xmin=0, xmax=70,grid=major, 
   	xlabel={Krylov subspace size $m$}, ylabel={2-norm error}]
		\addplot+[] 
			table[header=false, x expr=\lineno+1, y index={0}, row sep=\\]
		   {1.6044e+01\\ 1.3096e+01\\ 1.0164e+01\\ 7.5627e+00\\ 5.9632e+00\\ 
		   	4.7109e+00\\ 3.5769e+00\\ 2.4568e+00\\ 1.8382e+00\\ 1.5181e+00\\ 
		   	1.2882e+00\\ 1.1227e+00\\ 9.1526e-01\\ 7.2532e-01\\ 5.7283e-01\\ 
		   	4.3769e-01\\ 3.5475e-01\\ 2.7620e-01\\ 2.1168e-01\\ 1.6995e-01\\ 
		   	1.3264e-01\\ 9.4642e-02\\ 7.2317e-02\\ 5.7339e-02\\ 4.5401e-02\\ 
		   	3.3166e-02\\ 2.3847e-02\\ 1.5799e-02\\ 9.2018e-03\\ 5.6975e-03\\ 
		   	3.6427e-03\\ 2.2179e-03\\ 1.6274e-03\\ 1.1775e-03\\ 7.8311e-04\\ 
		   	4.1674e-04\\ 2.0364e-04\\ 1.0450e-04\\ 6.6016e-05\\ 4.4908e-05\\ 
		   	3.2954e-05\\ 2.1697e-05\\ 1.3733e-05\\ 8.2474e-06\\ 4.6140e-06\\ 
		   	2.4059e-06\\ 1.3798e-06\\ 7.5967e-07\\ 4.6182e-07\\ 2.4670e-07\\ 
		   	1.5673e-07\\ 9.1481e-08\\ 5.9349e-08\\ 3.2524e-08\\ 1.6489e-08\\ 
		   	8.3656e-09\\ 4.1213e-09\\ 1.7369e-09\\ 6.9614e-10\\ 3.6510e-10\\ 
		   	1.4597e-10\\ 6.4778e-11\\ 2.8298e-11\\ 1.2701e-11\\ 5.6105e-12\\ };
	\addlegendentry{error}
		\addplot+[] 
			table[header=false, x expr=\lineno+1, y index={0}, row sep=\\]
		   {1.0349e+02\\ 8.4676e+01\\ 6.9280e+01\\ 5.6684e+01\\ 4.6378e+01\\ 
		   	3.7945e+01\\ 3.1046e+01\\ 2.5401e+01\\ 2.0783e+01\\ 1.7004e+01\\ 
		   	1.3913e+01\\ 1.1383e+01\\ 9.3134e+00\\ 7.6200e+00\\ 6.2346e+00\\ 
		   	5.1010e+00\\ 4.1736e+00\\ 3.4147e+00\\ 2.7939e+00\\ 2.2859e+00\\ 
		   	1.8703e+00\\ 1.5302e+00\\ 1.2520e+00\\ 1.0244e+00\\ 8.3812e-01\\ 
		   	6.8573e-01\\ 5.6105e-01\\ 4.5904e-01\\ 3.7558e-01\\ 3.0729e-01\\ 
		   	2.5142e-01\\ 2.0571e-01\\ 1.6831e-01\\ 1.3771e-01\\ 1.1267e-01\\ 
		   	9.2184e-02\\ 7.5423e-02\\ 6.1710e-02\\ 5.0490e-02\\ 4.1310e-02\\ 
		   	3.3799e-02\\ 2.7654e-02\\ 2.2626e-02\\ 1.8512e-02\\ 1.5146e-02\\ 
		   	1.2392e-02\\ 1.0139e-02\\ 8.2957e-03\\ 6.7874e-03\\ 5.5533e-03\\ 
		   	4.5436e-03\\ 3.7175e-03\\ 3.0416e-03\\ 2.4886e-03\\ 2.0361e-03\\ 
		   	1.6659e-03\\ 1.3630e-03\\ 1.1152e-03\\ 9.1243e-04\\ 7.4653e-04\\ 
		   	6.1080e-04\\ 4.9975e-04\\ 4.0888e-04\\ 3.3454e-04\\ 2.7371e-04\\ };
	\addlegendentry{a priori bound \eqref{eq:apriori_bounds}}
    \end{semilogyaxis}
\end{tikzpicture}
 \caption{Convergence v.s.\ a priori bound \eqref{eq:apriori_bounds}, for Example~\ref{ex:invsqrt_apriori}.}
 \label{fig:apriori_invsqrt}
 \end{figure}
 We shall now illustrate the bound from
 Theorem~\ref{the:convergence_stieltjes} using a small numerical
 experiment. Let $A \in \mathbb{R}^{n \times n}, n = 100$ be a
 diagonal matrix with equidistantly spaced eigenvalues in $[0.1,10]$
 and take $\vy \in \mathbb{R}^n$ and $\vz \in
 \mathbb{R}^n$ at random. Figure~\ref{fig:apriori_invsqrt} shows the convergence
 of the approximation \eqref{eq:arnoldi_approximation_frechet} v.s.\
 the bound given by Theorem~\ref{the:convergence_stieltjes}. At the
 beginning, the convergence slope is captured very accurately, but due
 to the nature of the bound~\eqref{eq:convergence_bound_stieltjes}, it
 cannot predict the superlinear convergence occurring in later
 iterations due to spectral adaption. This is a typical shortcoming of
 many similar bounds for Stieltjes matrix functions.\hfill$\diamond$
\end{example}

\begin{remark}
In the proof of Theorem~\ref{the:convergence_stieltjes}, we have used the simple worst case upper bound for the CG error, as this gives rise to a simple, a priori bound for the error in the approximation of the \Fd. Of course, any other upper bound for the error in the CG method could be used in the same manner, and in particular one can expect superlinear convergence of the approximation $L_m^\Arn$ whenever superlinear convergence occurs for $\vx_m(t)$ and $\vu_m(t)$.\hfill$\diamond$
\end{remark}

To also brief\/ly touch on rational Krylov subspaces, we conclude this section by showing a simple result on the speed of convergence for the extended Krylov case (i.e., a rational Krylov subspace in which the poles $\xi_i$ are alternatingly chosen at $0$ and $\infty$). This result uses a similar approach as the one used in the proof of Theorem~\ref{the:convergence_stieltjes}. More refined results could be obtained by using techniques similar to those applied in~\cite{KnizhnermanSimoncini2011} to the case of the Lyapunov equation, but this is far beyond the scope of this paper.

\begin{theorem}\label{the:convergence_stieltjes_extended}\em 
Let $A \in \Cnn$ be Hermitian positive definite, let $\vy,\vz \in \Cn$ with $\|\vy\| = \|\vz\| = 1$, let $f$ be a Stieltjes function~\eqref{eq:stieltjes_integral} and let the poles in the rational Arnoldi method be chosen as $\xi_{2i-1} = \infty, \xi_{2i} = 0$, $i = 1,\dots,m$. Then the rational Krylov iterates $L_{m}^\Rat$ satisfy
\begin{equation}\label{eq:convergence_bound_stieltjes_extended}
\|L_f(A,\eta\vy\vz^H) - L_{m}^\Rat\| \leq |\eta|f(\lmin)C\left(\frac{\sqrt[4]{\kappa}-1}{\sqrt[4]{\kappa}+1}\right)^m
\end{equation}
where $\lmin$ is the smallest eigenvalue of $A$, $\kappa$ denotes the Euclidean norm condition number of $A$ and $C > 0$ is a constant that is independent of $m$ and $n$.
\end{theorem}
\begin{proof}
Similarly to the proof of Theorem~\ref{the:convergence_stieltjes}, we arrive at
\begin{equation}
\|L_f(A,\eta\vy\vz^H)-L_{m}^\Rat\| \leq |\eta|\int_0^\infty \|\vx(t)\|\|\vu(t)-\vu_{m}^\Rat(t)\| + \|\vu_{m}(t)\|\|\vx(t)-\vx_{m}^\Rat(t)\| \dmu \label{eq:stieltjes_convergence_extended_proof1}
\end{equation}
where now, $\vx_{m}^\Rat(t)$ and $\vu_{m}^\Rat(t)$ denote the rational Arnoldi approximations for the solutions of the shifted linear systems~\eqref{eq:shifted_systems_solutions}. We again have the estimates
$\|\vx(t)\| \leq \frac{1}{\lmin+t}$ and $\|\vu_{m}(t)\| \leq \frac{1}{\lmin+t}$. In addition, by using the fact that the resolvent is a Stieltjes function, we can employ a result of Beckermann and Reichel~\cite[Section~6.1]{BR09} to estimate
\begin{equation}\label{eq:knizhnerman_simoncini_bound}
\|\vx(t)-\vx_{m}^\Rat(t)\|, \|\vu(t)-\vu_{m}^\Rat(t)\| \leq \frac{C}{|\Phi_t(\sqrt{(\lmax+t)(\lmin+t)})|^m}
\end{equation}
where $C > 0$ is a constant that is independent of $t, m$ and $n$\footnote{We remark that the constant in the result of~\cite{BR09} does indeed depend on the spectral interval of the matrix $A+tI$ and thus on $t$, but as it is bounded from above for $t \in [0,\infty]$, we can replace it by a constant that is independent of $t$.} and $\Phi$ is the scaled inverse Zhukovsky function 
$$\Phi_t(z) = \frac{z-\gamma(t)}{\delta} + \sqrt{\left(\frac{z-\gamma(t)}{\delta}\right)^2-1}$$
with
$$\gamma(t) = \frac{\lmin+\lmax+2t}{2} \text{ and } \delta = \frac{\lmax-\lmin}{2}.$$

Inserting~\eqref{eq:knizhnerman_simoncini_bound} together with the straight-forward estimates into~\eqref{eq:stieltjes_convergence_extended_proof1}, we find
\begin{equation}\label{eq:stieltjes_convergence_extended_proof2}
\|L_f(A,\eta\vy\vz^H)-L_{m}^\Rat\| \leq C|\eta|\int_0^\infty  \frac{1}{(\lmin+t)\cdot|\Phi_t(\sqrt{(\lmax+t)(\lmin+t)})|^m} \dmu.
\end{equation}
Now, we have that 
\begin{equation}\label{eq:zhukovsky_value}\nonumber
|\Phi_t(\sqrt{(\lmax+t)(\lmin+t)})| = |\zeta(t) + \sqrt{\zeta(t)^2-1}| \text{ with }\zeta(t) = \frac{\sqrt{\kappa(t)} + 1}{\sqrt{\kappa(t)}-1},
\end{equation}
which, after standard algebraic manipulations, yields
\begin{equation}\label{eq:zhukovsky_value2}
|\Phi_t(\sqrt{(\lmax+t)(\lmin+t)}|^{-m} = \left(\frac{\sqrt[4]{\kappa(t)}-1}{\sqrt[4]{\kappa(t)}+1}\right)^m.
\end{equation}
The right-hand side of~\eqref{eq:zhukovsky_value2} is clearly monotonically decreasing in $t$, so that we can bound it by the value
\begin{equation}\label{eq:zhukovsky_value3}
|\Phi_t(\sqrt{\lmax\lmin})|^{-m} = \left(\frac{\sqrt[4]{\kappa}-1}{\sqrt[4]{\kappa}+1}\right)^m.
\end{equation}
Inserting~\eqref{eq:zhukovsky_value3} into~\eqref{eq:stieltjes_convergence_extended_proof2} concludes the proof of the theorem. 
\end{proof}

 \begin{figure}
 \centering
\tikzsetnextfilename{apriori_invsqrt_extended}
\pgfplotsset{height=0.35\linewidth,width=0.85\linewidth,compat=1.10,every axis/.append style={legend style={/tikz/every even column/.append style={column sep=6pt}}}}
\pgfplotscreateplotcyclelist{list_apri}{%
color_peter1,line width=1pt, mark=none, solid\\
color_peter3,line width=1pt, mark=none, densely dotted\\
}
\noindent%
\begin{tikzpicture}[scale=1]%
    \begin{semilogyaxis}[legend style={at={(1,1.15)}}, 
   	anchor= north east, legend columns=5,cycle list name=list_apri, 
   	xmin=0, xmax=25,grid=major, 
   	xlabel={Extended Krylov subspace size $m$.}, ylabel={2-norm error}]
		\addplot+[] 
			table[header=false, x expr=2*(\lineno+1), y index={0}, row sep=\\]
		   { 4.6603e+00 \\ 9.8293e-01 \\ 1.3414e-01 \\ 1.7061e-02 \\ 2.1484e-03 \\ 2.7910e-04 \\ 3.2852e-05 \\ 4.3745e-06 \\ 4.8511e-07 \\ 4.6971e-08 \\ 4.2063e-09 \\ 3.4299e-10 \\};
	\addlegendentry{error}
		\addplot+[] 
			table[header=false, x expr=2*(\lineno+1), y index={0}, row sep=\\]
		   {5.3975e+00 \\ 1.4566e+00 \\ 3.9311e-01 \\ 1.0609e-01 \\ 2.8631e-02 \\ 7.7267e-03 \\ 2.0852e-03 \\ 5.6275e-04 \\ 1.5187e-04 \\ 4.0986e-05 \\ 1.1061e-05 \\ 2.9851e-06 \\};
	\addlegendentry{slope of a priori bound \eqref{eq:convergence_bound_stieltjes_extended}}
    \end{semilogyaxis}
\end{tikzpicture}
 \caption{Convergence v.s.\ a priori bound \eqref{eq:convergence_bound_stieltjes_extended}, for Example~\ref{ex:invsqrt_apriori_extended}.}
 \label{fig:apriori_invsqrt_extended}
 \end{figure}
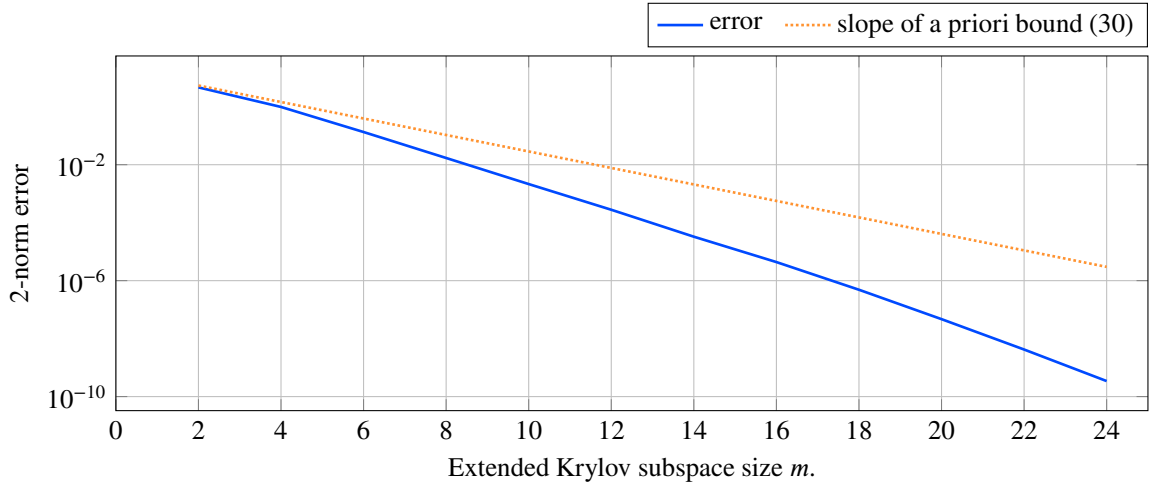

\begin{example}\label{ex:invsqrt_apriori_extended}
We now illustrate the bound from Theorem~\ref{the:convergence_stieltjes_extended} by a small numerical experiment, using the same setup as in Example~\ref{ex:invsqrt_apriori}. Figure~\ref{fig:apriori_invsqrt_extended} shows the convergence curve of the approximation \eqref{eq:rational_approximation_frechet} (with poles $\xi_{2i-1} = \infty, \xi_{2i} = 0$, $i = 1,\dots,m$) together with the bound from Theorem~\ref{the:convergence_stieltjes}. Note that we only give the slope of the bound, as the constant in~\eqref{eq:convergence_bound_stieltjes_extended} is not explicitly known. We observe that our a priori bound slightly overestimates the slope of the error norm reduction, but not by as much as in Example~\ref{ex:invsqrt_apriori}, as no superlinear convergence effects take place. \hfill$\diamond$
\end{example}

\section{The special case of the matrix logarithm}\label{sec:logarithm}
Another matrix function of interest, which does not fit into the framework considered so far, is the matrix logarithm $\log(A)$. While the logarithm cannot be represented by the Cauchy integral formula, we have the representation
\begin{equation}\label{eq:log_integral}\nonumber
\log(A) = \int_0^1 (A-I)\big(t(A-I)+I\big)^{-1}\d t,
\end{equation}
which holds for any $A$ having no eigenvalues on $\R^-$
(see e.g.~\cite{Higham2008}).
From this representation, we find an integral
representation of the \Fd\ as
\begin{equation}\label{eq:log_frechet_integral}\nonumber
L_{\log}(A,\eta\vy\vz^H) = \eta\int_0^1 \big(t(A-I)+I\big)^{-1}\vy\vz^H\big(t(A-I)+I\big)^{-1}\d t.
\end{equation}
Similar to what we outlined in section~\ref{sec:frechet_approx}, this is again the integral over outer products of solutions of two families of parameterized linear systems, i.e.,
\begin{equation}\label{eq:log_frechet_integral_compact}\nonumber
	L_{\log}(A,\eta\vy\vz^H) = \eta\int_0^1 \bar{\vx}(t)\bar{\vu}(t)^H \d t.
\end{equation}
where
\begin{equation}\label{eq:shifted_systems_solutions_log}\nonumber
\bar{\vx}(t) = \big(t(A-I)+I\big)^{-1}\vy \quad \text{and} \quad \bar{\vu}t) =\big(t(A-I)-I)\big)^{-H}\vz.
\end{equation}
Replacing $\bar{\vx}(t)$ and $\bar{\vu}(t)$ by their Arnoldi approximations
\begin{equation}\label{eq:arnoldi_approx_linearsystem_log}\nonumber
	\bar{\vx}_m^\Arn(t) = \|\vy\|V_m \big(t(G_m-I)+I\big)^{-1}\ve_1 \quad \text{and} \quad\bar{\vu}_m^\Arn(t) = \|\vz\|W_m \big(t(H_m-I)+I\big)^{-1}\ve_1
\end{equation}
then directly gives an Arnoldi approximation for the \Fd\ of the logarithm via
\begin{equation}
	\bar{L}_m^\Arn := \eta V_m \int_0^1\|\vy\|\|\vz\|\big(t(G_m-I)+I\big)^{-1}\ve_1\ve_1^H\big(t(H_m^H-I)+I\big)^{-1} \d t W_m^H =: V_m\bar{X}_mW_m^H.\label{eq:arnoldi_approx_log}
\end{equation}
A statement analogous to that of Lemma~\ref{lem:arnoldi_approx_block}
holds for the integral in~\eqref{eq:arnoldi_approx_log}, that is,
\begin{equation}\label{eq:block_log}
\log\left(\left[\begin{array}{cc}G_m & -\eta \|\vy\| \|\vz\| \ve_1\ve_1^H \\ 0 & H_m^H \end{array}\right]\right) = \left[\begin{array}{cc} \log(G_m) & \bar{X}_m \\ 0 & \log(H_m^H) \end{array}\right].
\end{equation}
When $A$ is Hermitian and $\vy = \vz$, we have (using the notation from section~\ref{subsec:lanczos})
\begin{equation}\label{eq:lanczos_approx_log}\nonumber
\bar{L}_m^\Lan := \eta V_m L_{\log}(T_m, \eta\|\vy\|^2 \ve_1\ve_1)V_m^H.
\end{equation}
Furthermore, in the Hermitian positive definite case, we can derive a convergence result for the logarithm which is very similar to the one for Stieltjes functions given in Theorem~\ref{the:convergence_stieltjes}.

\begin{theorem}\label{the:convergence_logarithm}\em
Let $A \in \Cnn$ be Hermitian positive definite and let $\vy,\vz \in \Cn$ with $\|\vy\| = \|\vz\| = 1$. Then the approximations $\Lhat_m^\Arn$ defined in~\eqref{eq:arnoldi_approx_log} satisfy
\begin{equation}\label{eq:convergence_bound_logarithm}\nonumber
	\|L_{\log}(A,\eta\vy\vz^H) - \bar{L}_m^\Arn\| \leq \frac{4 \eta}{\lmin} \left(\frac{\sqrt{\kappa}-1}{\sqrt{\kappa}+1}\right)^m,
\end{equation}
where $\lmin$ is the smallest eigenvalue of $A$ and $\kappa$ denotes the Euclidean norm condition number of $A$.
\end{theorem}
\begin{proof}
Proceeding analogously to the proof of Theorem~\ref{the:convergence_stieltjes}, we obtain
\begin{equation}\label{eq:log_convergence_proof1}\nonumber
	\|L_{\log}(A,\eta\vy\vz^H)-\bar{L}^\Arn_m\| \leq |\eta|\int_0^1 \|\bar{\vx}(t)\|\|\bar{\vu}(t)-\bar{\vu}_m^\Arn(t)\| + \|\bar{\vu}_m(t)^\Arn\|\|\bar{\vx}(t)-\bar{\vx}_m^\Arn(t)\| \d t.
\end{equation}
The smallest and largest eigenvalue and condition number of the matrices $\bar{A}(t) := t(A-I)+I$ are given by
$$\bar{\lambda}_{\min}(t) = t(\lmin-1)+1,\quad \bar{\lambda}_{\max}(t) = t(\lmax-1)+1,\quad \bar{\kappa}(t) = \frac{t(\lmax-1)+1}{t(\lmin-1)+1},$$
where $\lmin$ and $\lmax$ are the largest and smallest eigenvalue of $A$.

Using similar arguments as in the proof of Theorem~\ref{the:convergence_stieltjes} together with the relation
\begin{equation}\label{eq:log_norm_equivalence}\nonumber
\|\vv\| \leq \frac{1}{\sqrt{t(\lmin-1)+1}}\|\vv\|_{\widehat{A}(t)},
\end{equation}
we obtain
$$\|\bar{\vx}(t)-\bar{\vx}_m^\Arn(t)\| \leq \frac{2}{\sqrt{t(\lmin-1)+1}}\left(\frac{\sqrt{\bar{\kappa}(t)}-1}{\sqrt{\bar{\kappa}(t)}+1}\right)^m\|\bar{\vx}(t)\|_{\widehat{A}(t)}.$$
Using the estimates
$$\|\bar{\vx}(t)\|_{\bar{A}(t)} \leq \frac{1}{\sqrt{t(\lmin-1)+1}}, \quad \text{and} \quad \|\bar{\vx}(t)\| \leq \frac{1}{t(\lmin-1)},$$
(and analogous versions for $\bar{\vu}(t), \bar{\vu}_m^\Arn(t)$), we finally obtain
\begin{equation}\label{eq:log_convergence_proof3}\nonumber
	\|\vx(t)\|\|\vu(t)-\vu_m^\Arn(t)\| + \|\vu_m^\Arn(t)\|\|\vx(t)-\vx_m^\Arn(t)\| \leq \frac{4}{(t(\lmin-1)+1)^2}\left(\frac{\sqrt{\bar{\kappa}(t)}-1}{\sqrt{\bar{\kappa}(t)}+1}\right)^m.
\end{equation}
Combining this with~\eqref{eq:stieltjes_convergence_proof1}, we find
\begin{equation}\label{eq:log_convergence_proof4}\nonumber
	\|L_{\log}(A,\eta\vy\vz^H)-\bar{L}_m^\Arn\| \leq |\eta|\int_0^1  \frac{4}{((t(\lmin-1)+1)^2}\left(\frac{\sqrt{\bar{\kappa}(t)}-1}{\sqrt{\bar{\kappa}(t)}+1}\right)^m \d t.
\end{equation}
Now since $\bar{\kappa}(t)$ is monotonically increasing on $[0,1]$ we know that $\bar{\kappa}(1) = \kappa$ which, when combined with
$$\int_0^1 \frac{1}{(t(\lmin-1)+1)^2} \d t= \frac{1}{\lmin},$$
gives us the desired error bound.
\end{proof}

\section{A posteriori error estimates}\label{sec:posteriori}

In this section we derive a heuristic a posteriori error estimate for the Arnoldi approximation \eqref{eq:arnoldi_approximation_frechet}.
First, consider the Krylov subspace $\spK_m(A,\vy)$ and the basis matrix $V_m$ and the Hessenberg matrix $G_m$ given by the Arnoldi iteration, satisfying the relation \eqref{eq:arnoldi_relation1}.
A commonly used a posteriori error estimate (see e.g.~\cite[Sec.\;5.2]{Saad1992}) for the Arnoldi approximation of the matrix exponential is given by
$$
\| \exp(t A) \vy - V_m \exp(t G_m) \ve_1 \| \vy \|  \| \approx g_{m+1,1} \ve_m^H \varphi_1(t G_m) \ve_1 \| \vy \|,
$$
where $\varphi_1(z) = (e^z - 1)/z$. Using the Cauchy integral formula for $\varphi_1(z)$ 
(see~\cite[Thm.\;5.1]{SchmelzerTrefethen2007}), and choosing a contour $\Gamma$ 
which encircles $\{0\} \cup \spec(G_m)$, this estimate can be written as
\begin{equation} \label{eq:apost0}
	g_{m+1,1} \ve_m^H \varphi_1(t G_m) \ve_1 \| \vy \| =  \| \vy \| \int_\Gamma e^{t\lambda} \, \frac{g_{m+1,1}}{\lambda}  \, \ve_m^H  (\lambda I - G_m)^{-1} \ve_1 \d \lambda.
\end{equation}
Next, we consider an analytic function $f$ and the Krylov subspaces $\spK_m(A,\vy)$ and $\spK_m(A^H\!,\vz)$ and the bases
$V_m$ and $W_m$ and the corresponding Hessenberg matrices $G_m$ and $H_m$ satisfying the relations 
\eqref{eq:arnoldi_relation1} and \eqref{eq:arnoldi_relation2}.
Motivated by \eqref{eq:apost0}, we estimate the error of the Arnoldi approximation \eqref{eq:arnoldi_approximation_frechet}
of the \Fd\ $L_f(A, \eta \vy \vz^H)$ by a heuristic estimate
\begin{equation} \label{eq:apost1}\nonumber
\|L_f(A, \eta \vy \vz^H) - L_m^\Arn\| \approx  \eta \,  \| \vy \|\| \vz \|   
\int_\Gamma f(\lambda) \frac{g_{m+1,m} h_{m+1,m}}{\lambda^2}  \, \ve_m^H  (\lambda I - G_m)^{-1}  \ve_1 \ve_1^H (\lambda I - H_m^H)^{-1}   \ve_1 \d \lambda.
\end{equation}
This estimate can be evaluated using a $4 \times 4$-block matrix, because it follows from block Gaussian elimination that
$$
\lambda^{-2}(\lambda I - G_m)^{-1}  \ve_1 \ve_1^H (\lambda I - H_m^H)^{-1} = (\lambda I - \widetilde{G}_m)^{-1}_{1:m,3m+1:4m},
$$ 
where
$$
\widetilde{G}_m = \begin{bmatrix} G_m & - \ve_1 \ve_1^H & 0 & 0 \\ 0 & H_m^H & I & 0 \\ 0 & 0 & 0 & I \\ 0 & 0 & 0 & 0 \end{bmatrix}.
$$
Thus, 
for an analytic function $f$, we use as an estimate
\begin{equation} \label{eq:apost2}
\|L_f(A, \eta \vy \vz^H) - L_m^\Arn\|  \approx  g_{m+1,m} h_{m+1,m}  \ve_m^H \big( f(\widetilde{G}_m)_{1:m,3m+1:4m} \big) \ve_1 = g_{m+1,m} h_{m+1,m}  f(\widetilde{G}_m)_{m,3m+1}.
\end{equation}
Notice that this approach is not directly applicable to, e.g., the matrix logarithm, as the matrix $\widetilde{G}_m$ is singular.

\begin{figure}
\centering
\tikzsetnextfilename{aposteriori}
\pgfplotsset{height=0.35\linewidth,width=0.85\linewidth,compat=1.10,every axis/.append style={legend style={/tikz/every even column/.append style={column sep=6pt}}}}
\pgfplotscreateplotcyclelist{list_apost}{%
color_peter1,line width=1pt, mark=none, solid\\
color_peter3,line width=1pt, mark=none, densely dotted\\
color_peter2,line width=1pt, mark=none, dashed\\
}
\noindent%
\begin{tikzpicture}[scale=1]%
    \begin{semilogyaxis}[legend style={at={(1,1.15)}}, 
   	anchor= north east, legend columns=5,cycle list name=list_apost, 
   	xmin=0, xmax=50,grid=major, 
   	xlabel={Krylov subspace size $m$}, ylabel={2-norm error}]

\addplot+[]
table [x ={x},y ={error}] {fig/tikz_data/aposteriori.dat};\addlegendentry{error}
 
\addplot+[]
table [x ={x},y ={apost}]{fig/tikz_data/aposteriori.dat} node [pos=0,left] {}; \addlegendentry{estimate~\eqref{eq:apost2}}

\addplot+[]
table [x ={x},y ={diff}]{fig/tikz_data/aposteriori.dat} node [pos=0,left] {}; \addlegendentry{estimate~\eqref{eq:simple_est}}
    \end{semilogyaxis}
\end{tikzpicture}
 \caption{Convergence v.s.\ a the a posteriori estimates \eqref{eq:apost2} and \eqref{eq:simple_est}, for Example~\ref{eg:a_posteriori}.}
 \label{fig:aposteriori}
 \end{figure}
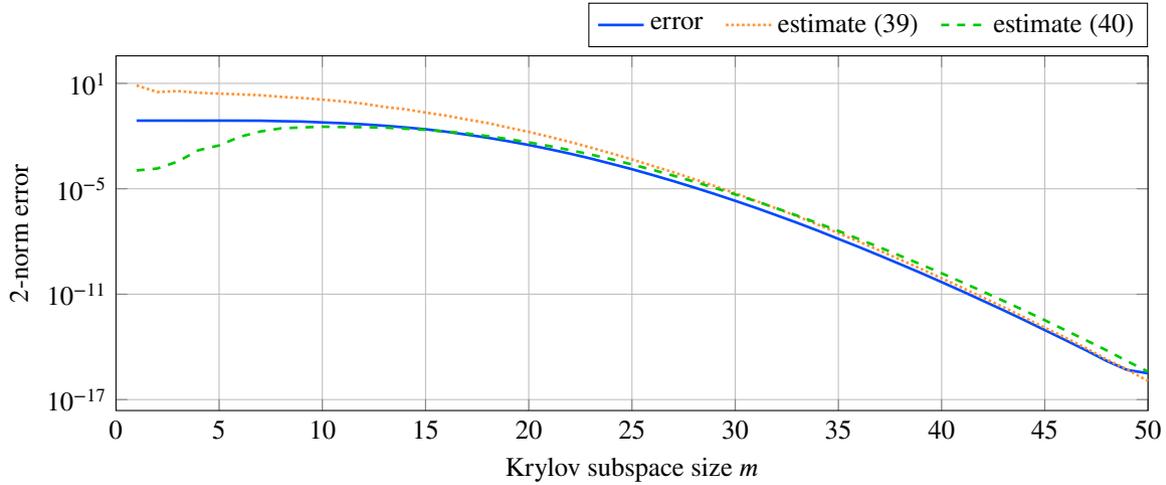

A simple and efficient estimate for the error $\|L_f(A, \eta \vy \vz^H) - L_m^\Arn\|$ can also be obtained by comparing subsequent Krylov subspace approximations.
This means that for $d \in \mathbb{Z}^+$, we estimate
\begin{equation} \label{eq:simple_est}
	\|L_f(A, \eta \vy \vz^H) - L_m^\Arn\| \approx \| L_{m+d}^\Arn  - L_m^\Arn\|.
\end{equation}
The estimate \eqref{eq:simple_est} can be evaluated using small dimensional matrices since
\begin{equation*}
\| L_{m+d}^\Arn - L_m^\Arn \|	=  \eta \cdot \| V_{m+d} X_{m+d} W_{m+d}^H - V_m X_m W_m^H \| =  \eta \cdot \| X_{m+d} - \begin{bmatrix} X_m & 0 \\ 0 & 0 \end{bmatrix} \|;
\end{equation*}
see also~\cite[Section 2.3]{BeckermannKressnerSchweitzer2018} and~\cite[Section 3]{Kressner2019}.

\begin{example}\label{eg:a_posteriori}
 Consider the following simple numerical example to illustrate the estimates \eqref{eq:apost2} and \eqref{eq:simple_est}.
 Set $t=5$, $A = \mathrm{diag}(1,\ -2,\ 1)  +  1.5 \cdot \mathrm{diag}(-1,0,\ 1)  \in \mathbb{R}^{n \times n}$, and take randomly $\vy \in \mathbb{R}^n$
 and $\vz \in \mathbb{R}^n$. Set $n=100$. Figure~\ref{fig:aposteriori} shows the actual convergence of the approximation \eqref{eq:arnoldi_approximation_frechet} 
	and the estimate \eqref{eq:apost2} and the estimate \eqref{eq:simple_est} for $d=1$. For later iterations, both estimates are very accurate, while for early iterations, the estimate~\eqref{eq:apost2} overestimates the actual error norm, while~\eqref{eq:simple_est} underestimates it. In particular in situations where it is crucial to reach a certain accuracy, it is advisable to be careful when using estimate~\eqref{eq:simple_est} as stopping criterion as it might severly underestimate the actual error when convergence is slow.\hfill$\diamond$
\end{example}

\section{Numerical experiments}\label{sec:experiments}
In this section we will compare our algorithms against alternatives in the literature in a number of different scenarios.
All experiments in this section are run on a
Linux machine running MATLAB 2016b.
In order to increase the reliability of the timings we use only a
single core,
and run MATLAB with no GUI (using the \texttt{--nojvm} option).

In our first experiment we compare the different proposed algorithms to each other for two simple model problems. Next, we compare our new algorithms against alternatives when computing $L_f(A,E)\vb$ over a set of
difficult test problems. Finally, we compare our new algorithms to existing alternatives when computing
$L_f(A,E)$ in the context of a physics application:
obtaining the sensitivity of nuclear activation and transmutation
to the system input parameters.


\subsection{Comparison of our methods for simple model problems}

\begin{figure}[t]
\centering
\tikzsetnextfilename{laplace}
\pgfplotsset{height=0.35\linewidth,width=0.85\linewidth,compat=1.10,every axis/.append style={legend style={/tikz/every even column/.append style={column sep=6pt}}}}
\pgfplotscreateplotcyclelist{list_laplace}{%
color_peter1,line width=1pt, mark=none, solid\\
color_peter3,line width=1pt, mark=none, densely dotted\\
}
\noindent%
\begin{tikzpicture}[scale=1]%
    \begin{semilogyaxis}[legend style={at={(1,1.15)}}, 
   	anchor= north east, legend columns=5,cycle list name=list_laplace, 
   	xmin=0, xmax=70,grid=major, 
   	xlabel={Krylov subspace size $m$}, ylabel={2-norm error}]

\addplot[color=ForestGreen,thick]
table [x ={x},y ={err_poly}] {fig/tikz_data/laplace2d.dat};\addlegendentry{Lanczos}
 
  \addplot[color=NavyBlue,thick]
table [x ={x},y ={err_block}]{fig/tikz_data/laplace2d.dat} node [pos=0,left] {}; \addlegendentry{Block}

  \addplot[color=Magenta,thick]
table [x ={x},y ={err_extended}]{fig/tikz_data/laplace2d.dat} node [pos=0,left] {}; \addlegendentry{Extended}

  \addplot[color=ForestGreen,dashed,thick]
table [x ={x},y ={bound_poly}]{fig/tikz_data/laplace2d.dat} node [pos=0,left] {}; \addlegendentry{Slope of~\eqref{eq:convergence_bound_stieltjes}}

  \addplot[color=Magenta,dashed,thick]
table [x ={x},y ={bound_extended}]{fig/tikz_data/laplace2d.dat} node [pos=0,left] {}; \addlegendentry{Slope of~\eqref{eq:convergence_bound_stieltjes_extended}}
    \end{semilogyaxis}
\end{tikzpicture}
\caption{Error norm and error bounds when approximating $L_f(A,E)$ by several of our proposed methods where $A$ is the discretization of the two-dimensional Laplace operator, $f(z) = z^{-1/2}$ and $E$ is a random rank one matrix.}
\label{fig:2dlaplace}
\end{figure}
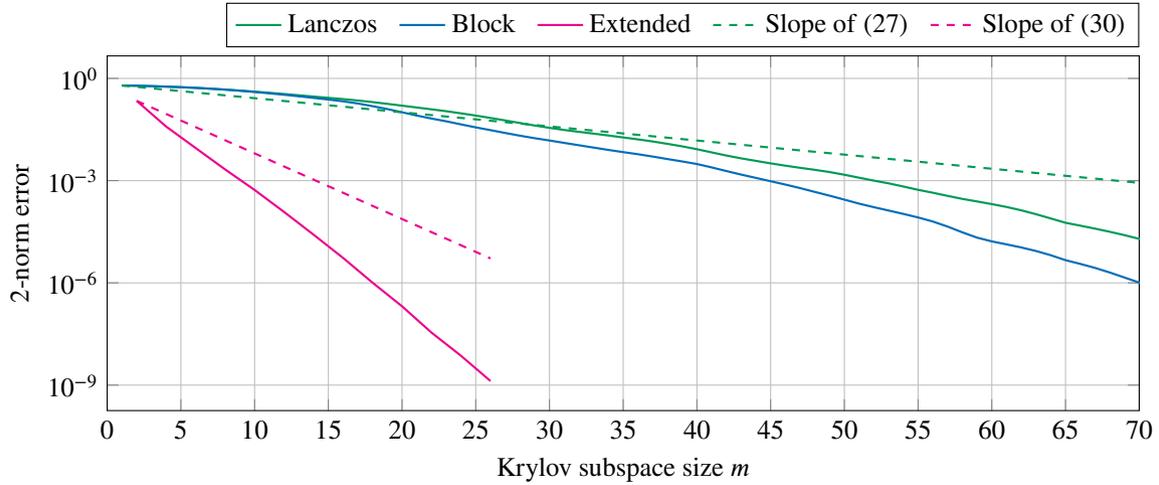

We begin by performing two simple experiments, one involving a Hermitian matrix and one involving a non-Hermitian matrix, in order to compare all the different methods that we proposed in Section~\ref{sec:frechet_approx}. First, let $A \in \Cnn$, $n = 32^2$ be the Hermitian positive definite matrix corresponding to the discretization of the two-dimensional Laplace equation on a square grid with Dirichlet boundary conditions. We consider a rank-one direction term $E = \vy\vz^H$, i.e., $\eta = 1$, where $\vy$ and $\vz$ are random vectors of unit norm and the inverse square root function $f(z) = z^{-1/2}$. We compare the basic polynomial Krylov method from Algorithm~\ref{alg:arnoldi_frechet} (where the Arnoldi process is replaced by the Lanczos process) to the block Lanczos method presented in Section~\ref{subsec:block} and the extended Krylov method from Section~\ref{subsec:rational} and aim for an approximation error below $10^{-8}$. The error norms of the corresponding approximations are given in Figure~\ref{fig:2dlaplace} together with the slopes of the convergence bounds from Theorem~\ref{the:convergence_stieltjes} and~\ref{the:convergence_stieltjes_extended}. The Lanczos method reaches the desired accuracy after 86 iterations, while the block Lanczos approach requires 74 iterations (i.e., about 15\% less than the standard Lanczos method), showcasing the larger approximation power of block Krylov spaces. As is expected, the extended Krylov method converges fastest in terms of subspace dimension, finding an accurate approximation in a space of dimension 26, but requires 13 linear system solves with $A$. Concerning the quality of our convergence estimates, it can be observed that the bound~\eqref{eq:convergence_bound_stieltjes} quite accurately predicts the slope of the real error norm, but of course fails to predict the superlinear convergence caused by spectral adaptation in later iterations. The bound~\eqref{eq:convergence_bound_stieltjes_extended} for the extended Krylov subspace method predicts convergence that is much faster than that of the polynomial methods but overestimates the actual slope by quite some margin. To gauge the approximation power of the extended Krylov subspace, from which we find a rank $26$ approximation of $L_f(A,E)$, we also plot the singular values of $L_f(A,E)$ in Figure~\ref{fig:sv} (left). This plot reveals that it is (in theory) possible to approximate $L_f(A,E)$ to accuracy $10^{-8}$ by a matrix of rank $9$. It is of course unrealistic to expect an iterative Krylov method to find this optimal low-rank approximation (in particular with non-optimized poles).

\begin{figure}[t]
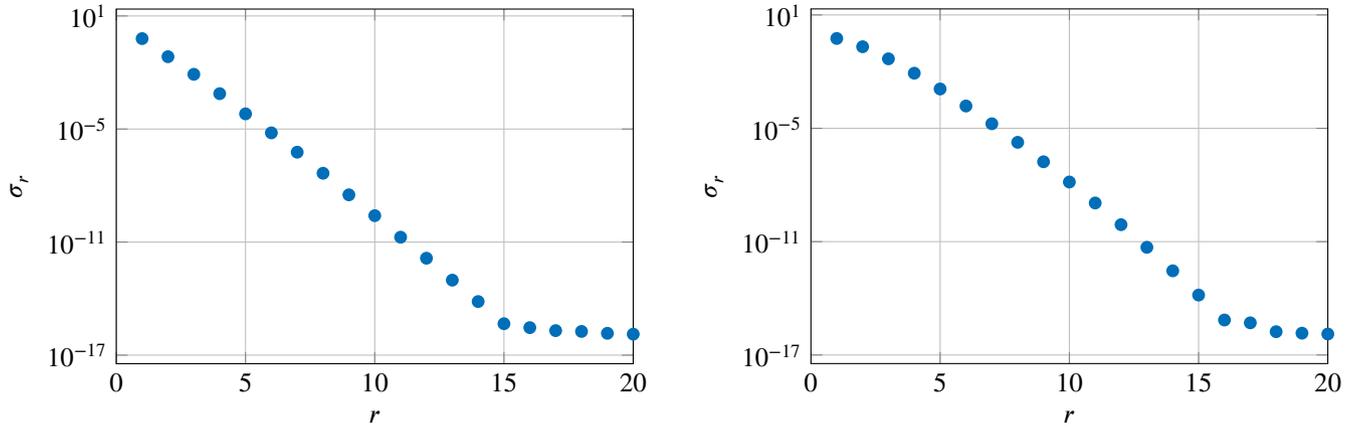

\centering
\input{./fig/experiment_laplace_sv.tikz}\hfill\input{./fig/experiment_convdiff_sv.tikz}
\caption{The 20 largest singular values of the Fr\'echet derivative $L_f(A,E)$ where (left) $A$ is the discretization of the two-dimensional Laplace operator and $f(z) = z^{-1/2}$ and (right) $A$ is the discretization of a two-dimensional convection diffusion operator and $f(z) = \exp(-tz)$. In both cases, $E$ is a random rank one matrix.}
\label{fig:sv}
\end{figure}

For testing the methods geared towards non-Hermitian problems, in particular the short-recurrence two-sided Lanczos method, we perform a similar experiment as before, but this time consider $A$ stemming from a semi-discretization of the following two-dimensional convection diffusion equation
\begin{eqnarray}
	\frac{\partial u}{\partial t} - \Delta u + \tau_1 \frac{\partial u}{\partial x_1} + \tau_2 \frac{\partial u}{\partial x_2} &=& 0\phantom{(x_0)} \text{ on } (0,1)^2 \times (0,T), \nonumber\\
u(x,t) &=& 0\phantom{(x_0)}\text{ on } \partial(0,1)^2 \text{ for all } t \in [0,T], \nonumber\label{eq:convdiff_equation} \\
u(x,0) &=& u_0(x) \text{ for all } x \in (0,1)^2.\nonumber
\end{eqnarray}
In particular, using central differences with uniform discretization step size $h$ for the differential operator $-\Delta u + \tau_1 \frac{\partial u}{\partial x_1} + \tau_2 \frac{\partial u}{\partial x_2}$ yields the matrix
\begin{equation}\label{eq:convdiff_discretization}
A = -\frac{1}{h^2} \left(I \otimes C_1 + C_2 \otimes I\right) \in \R^{n^2 \times n^2}
\end{equation}
with
$$C_i = \left[\begin{array}{ccccc}
-2 & 1-\frac{\tau_i h}{2} & & & \\
1+\frac{\tau_i h}{2} & -2 &1-\frac{\tau_i h}{2} & &\\
 & 1+\frac{\tau_i h}{2} & \ddots & \ddots &\\
 & & \ddots & \ddots & 1-\frac{\tau_i h}{2} \\
 & & & 1+\frac{\tau_i h}{2} & -2
\end{array}\right] \in \Rnn, i = 1,2.$$
The convection coefficients $\tau_i, i = 1,2$ are chosen such that the P{\'e}clet numbers $\text{Pe}_i = \frac{\tau_i h}{2}$ are equal to $\text{Pe}_1 = .5$ and $\text{Pe}_2 = .25$, respectively. We aim to approximate $L_f(A,E)$, where $f(z) =\exp(-tz)$ for a time step $t = .005$ and $E$ is a random rank one matrix. We compare the standard Arnoldi method, Algorithm~\ref{alg:arnoldi_frechet}, the two-sided Lanczos method from Section~\ref{subsec:twosided} and a shift-and-invert Krylov method, i.e., a rational Krylov method with a single repeated pole. As all eigenvalues of $A$ from~\eqref{eq:convdiff_discretization} are real and positive, we heuristically choose the shift $\xi = \sqrt{\lmin\lmax}$, a choice that is often employed in the Hermitian case.

The results of this experiment are depicted in Figure~\ref{fig:2dconvdiff}, and the largest singular values of $L_f(A,E)$ are given in Figure~\ref{fig:sv} (right). The Arnoldi method and two-sided Lanczos method require roughly the same subspace dimension for reaching the target accuracy, but the convergence curve of the two-sided method is very nonsmooth compared to that of the Arnoldi method, and in particular non-monotonic. The shift-and-invert method requires a little more than half the subspace dimension of the polynomial methods and produces a rank-$22$ approximation of $L_f(A,E)$. From the singular values, it can be seen that the best possible approximation reaching the target accuracy has rank $10$.

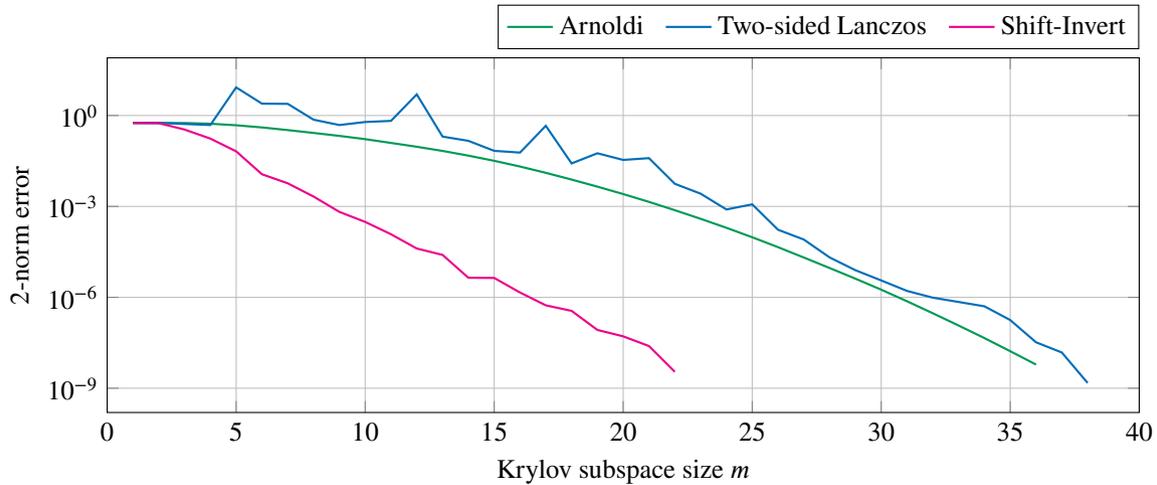
\begin{figure}[t]
\centering
\tikzsetnextfilename{convdiff}
\pgfplotsset{height=0.35\linewidth,width=0.85\linewidth,compat=1.10,every axis/.append style={legend style={/tikz/every even column/.append style={column sep=6pt}}}}
\pgfplotscreateplotcyclelist{list_laplace}{%
color_peter1,line width=1pt, mark=none, solid\\
color_peter3,line width=1pt, mark=none, densely dotted\\
}
\noindent%
\begin{tikzpicture}[scale=1]%
    \begin{semilogyaxis}[legend style={at={(1,1.15)}}, 
   	anchor= north east, legend columns=5,cycle list name=list_laplace, 
   	xmin=0, xmax=40,grid=major, 
   	xlabel={Krylov subspace size $m$}, ylabel={2-norm error}]

\addplot[color=ForestGreen,thick]
table [x ={x},y ={err_poly}] {fig/tikz_data/convdiff2d.dat};\addlegendentry{Arnoldi}
 
  \addplot[color=NavyBlue,thick]
table [x ={x},y ={err_twosided}]{fig/tikz_data/convdiff2d.dat} node [pos=0,left] {}; \addlegendentry{Two-sided Lanczos}

  \addplot[color=Magenta,thick]
table [x ={x},y ={err_si}]{fig/tikz_data/convdiff2d.dat} node [pos=0,left] {}; \addlegendentry{Shift-Invert}
    \end{semilogyaxis}
\end{tikzpicture}
\caption{Error norm when approximating $L_f(A,E)$ by several of our proposed methods where $A$ is the discretization of a two-dimensional convection diffusion operator,  $f(z) =\exp(-tz)$ and $E$ is a random rank one matrix.}
\label{fig:2dconvdiff}
\end{figure}

\subsection{Accurate computation of $L_f(A,E)\vb$}\label{sec:accuracy}

In our next experiment we will compare the accuracy of three
competing algorithms when aiming to approximate
$L_f(A,E)\vb$,
where $f(z) = e^z$ and $A$ is a matrix taken from
the Matrix Computation Toolbox\cite{high-mct}.
This toolbox contains a selection of difficult test matrices,
i.e. matrices that are known to be ill-conditioned or have
ill-conditioned eigenvalues.
The matrices $A$ from the toolbox are scaled to have unit $2$-norm
whilst $E$ and $b$ have elements drawn from a Normal $N(0,1)$ distribution.
Note that the condition number of $L_f(A,E)\vb$ is not clearly related to
the condition number of $A$ itself,
so even this scaling resulted in some matrices that were too
ill-conditioned to return sensible results.

There are three algorithms that we compare to one another
within this section.
The first is our algorithm~\ref{alg:arnoldi_frechet};
to which we make a minor modification by multiplying the result
by the vector $b$,
and iteratively increasing the rank $m$ until the relative
difference between two iterates is less than the desired tolerance.
The second algorithm we consider is taken directly
from Kandolf and Relton~\cite{KandolfRelton2016}.
They use a Krylov subspace approach to approximate
$L_f(A,E)\vb$ directly, without forming $L_f(A,E)$.
We will call this the ``KR algorithm''.
Our final algorithm is multiplying the result of the $2 \times 2$ block approach
in equation~\eqref{eq:2x2block} by the vector $[0, b]^T$,
from which we can obtain $L_f(A,E)\vb$ as the upper half of the
resulting vector.
By computing this latter vector using the MATLAB function
\texttt{expmv} (by Al-Mohy and Higham\cite{alhi11})
we do not need to form $L_f(A,E)$ in full.

We take as the ``exact'' answer, for comparison,
the result obtained by applying the
$2 \times 2$ block algorithm using $100$ digit arithmetic,
making use of the Symbolic Math Toolbox in MATLAB.

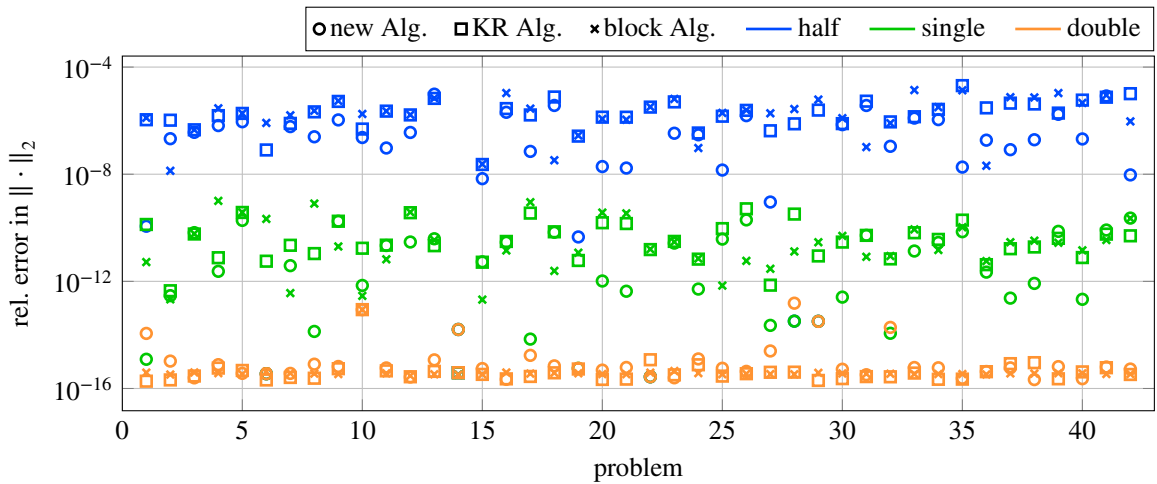
\begin{figure}[t]
\centering
\tikzsetnextfilename{accuracy_testmatrices}
\pgfplotsset{height=0.35\linewidth,width=0.85\linewidth,compat=1.10,every axis/.append style={legend style={/tikz/every even column/.append style={column sep=6pt}}}}
\pgfplotscreateplotcyclelist{list_atm}{%
color_peter1,line width=1pt, mark=o, solid\\
color_peter1,line width=1pt, mark=square, solid\\
color_peter1,line width=1pt, mark=x, solid\\
color_peter2,line width=1pt, mark=o, solid\\
color_peter2,line width=1pt, mark=square, solid\\
color_peter2,line width=1pt, mark=x, solid\\
color_peter3,line width=1pt, mark=o, solid\\
color_peter3,line width=1pt, mark=square, solid\\
color_peter3,line width=1pt, mark=x, solid\\
}

\noindent%
\begin{tikzpicture}[scale=1]%
    \begin{customlegend}[legend style={at={(0,0.65)}}, anchor= north east, 
 		legend columns=6,
 		legend style={/tikz/every even column/.append style={column sep=0.3cm}},
 		legend entries={new Alg., KR Alg., block Alg., half,single,double}]
        \addlegendimage{black,line width=1pt, mark=o, only marks}
        \addlegendimage{black,line width=1pt, mark=square, only marks}
        \addlegendimage{black,line width=1pt, mark=x, only marks}
        \addlegendimage{color_peter1,line width=1pt, solid, sharp plot}
        \addlegendimage{color_peter2,line width=1pt, solid,sharp plot}
        \addlegendimage{color_peter3,line width=1pt, solid,sharp plot}
   	\end{customlegend}
   	\begin{semilogyaxis}[legend style={at={(1,1.15)}}, 
   	anchor= north east, legend columns=5,cycle list name=list_atm, 
   	xmin=0, xmax=43,grid=major, 
   	xlabel={problem}, ylabel={rel. error in $\|\cdot\|_2$},
   	]
		\addplot+[only marks,] 
			table[header=false, x expr=\lineno+1, y index={0}, row sep=\\]
		   {1.1441e-10\\ 2.1168e-07\\ 3.6587e-07\\ 6.5453e-07\\ 
		   	9.1077e-07\\ 3.5584e-16\\ 5.9898e-07\\ 2.4960e-07\\ 1.0649e-06\\ 
		   	2.3452e-07\\ 9.5225e-08\\ 3.6216e-07\\ 9.7408e-06\\ 1.5966e-14\\ 
		   	6.8156e-09\\ 2.0514e-06\\ 7.1074e-08\\ 3.6933e-06\\ 4.5475e-11\\ 
		   	1.9305e-08\\ 1.6992e-08\\ 2.7423e-16\\ 3.3841e-07\\ 2.9639e-07\\ 
		   	1.4294e-08\\ 1.5357e-06\\ 9.1394e-10\\ 3.2690e-14\\ 
		   	3.2170e-14\\ 7.0900e-07\\ 3.6832e-06\\ 1.1016e-07\\ 1.2459e-06\\ 
		   	1.0921e-06\\ 1.8441e-08\\ 1.8648e-07\\ 8.2403e-08\\ 1.9355e-07\\ 
		   	1.7227e-06\\ 2.0707e-07\\ 8.3306e-06\\ 9.4224e-09\\ };

		\addplot+[only marks] 
			table[header=false, x expr=\lineno+1, y index={0}, row sep=\\]
		   {1.0817e-06\\ 1.0385e-06\\ 4.5547e-07\\ 1.5379e-06\\ 
		   	1.8611e-06\\ 8.0393e-08\\ 7.8946e-07\\ 2.1302e-06\\ 5.2241e-06\\ 
		   	4.8539e-07\\ 2.2661e-06\\ 1.6365e-06\\ 6.7786e-06\\ 3.8836e-16\\ 
		   	2.3062e-08\\ 2.8244e-06\\ 1.6326e-06\\ 7.6306e-06\\ 2.6436e-07\\ 
		   	1.3488e-06\\ 1.3443e-06\\ 3.2028e-06\\ 5.0091e-06\\ 3.4580e-07\\ 
		   	1.4640e-06\\ 2.4221e-06\\ 4.1873e-07\\ 7.6162e-07\\ 
		   	2.4751e-06\\ 7.7553e-07\\ 5.3327e-06\\ 8.9841e-07\\ 1.4005e-06\\ 
		   	2.6365e-06\\ 2.0303e-05\\ 2.9987e-06\\ 4.5008e-06\\ 4.1864e-06\\ 
		   	1.9143e-06\\ 5.8402e-06\\ 7.5000e-06\\ 1.0161e-05\\};

		\addplot+[only marks] 
			table[header=false, x expr=\lineno+1, y index={0}, row sep=\\]
		   {1.2789e-06\\ 1.3413e-08\\ 4.5772e-07\\ 2.8869e-06\\ 
		   	1.8626e-06\\ 8.2207e-07\\ 1.5840e-06\\ 2.3424e-06\\ 5.4739e-06\\ 
		   	1.7774e-06\\ 2.2647e-06\\ 1.6585e-06\\ 6.8995e-06\\ 3.2844e-16\\ 
		   	2.3829e-08\\ 1.0730e-05\\ 2.8094e-06\\ 3.3262e-08\\ 2.7879e-07\\ 
		   	1.1607e-06\\ 1.0842e-06\\ 3.2579e-06\\ 6.4702e-06\\ 9.5031e-08\\ 
		   	1.9436e-06\\ 2.4592e-06\\ 1.8820e-06\\ 2.7141e-06\\ 
		   	6.1152e-06\\ 1.2362e-06\\ 1.0299e-07\\ 8.0343e-07\\ 1.3845e-05\\ 
		   	3.1832e-06\\ 1.3593e-05\\ 2.0682e-08\\ 7.5154e-06\\ 7.3428e-06\\ 
		   	1.0717e-05\\ 4.5940e-06\\ 8.9374e-06\\ 9.3788e-07\\ };

		\addplot+[only marks] 
			table[header=false, x expr=\lineno+1, y index={0}, row sep=\\]
		   {1.2316e-15\\ 2.9194e-13\\ 6.5840e-11\\ 2.3706e-12\\ 
		   	1.8878e-10\\ 3.5584e-16\\ 3.8570e-12\\ 1.3495e-14\\ 1.7527e-10\\ 
		   	7.0317e-13\\ 2.1808e-11\\ 2.9451e-11\\ 3.8873e-11\\ 1.5966e-14\\ 
		   	5.3542e-12\\ 2.8885e-11\\ 6.9776e-15\\ 6.7832e-11\\ 5.6535e-16\\ 
		   	1.0306e-12\\ 4.2743e-13\\ 2.7423e-16\\ 2.6789e-11\\ 5.1512e-13\\ 
		   	3.7832e-11\\ 1.9708e-10\\ 2.2702e-14\\ 3.2690e-14\\ 
		   	3.2170e-14\\ 2.5770e-13\\ 5.2184e-11\\ 1.1513e-14\\ 1.3558e-11\\ 
		   	2.8262e-11\\ 7.1371e-11\\ 2.2084e-12\\ 2.3628e-13\\ 8.3750e-13\\ 
		   	7.3512e-11\\ 2.1362e-13\\ 8.2353e-11\\ 2.2985e-10\\ };

		\addplot+[only marks] 
			table[header=false, x expr=\lineno+1, y index={0}, row sep=\\]
		   {1.3279e-10\\ 4.4529e-13\\ 5.8598e-11\\ 7.6220e-12\\ 
		   	3.7342e-10\\ 5.6525e-12\\ 2.2207e-11\\ 1.0891e-11\\ 1.7583e-10\\ 
		   	1.7226e-11\\ 2.2446e-11\\ 3.6541e-10\\ 2.1432e-11\\ 3.8836e-16\\ 
		   	5.1795e-12\\ 3.0872e-11\\ 3.5299e-10\\ 7.0862e-11\\ 6.0155e-12\\ 
		   	1.5441e-10\\ 1.4389e-10\\ 1.5351e-11\\ 3.1043e-11\\ 6.7515e-12\\ 
		   	9.3129e-11\\ 5.1128e-10\\ 7.2025e-13\\ 3.2480e-10\\ 
		   	8.8761e-12\\ 2.9240e-11\\ 5.2109e-11\\ 6.8947e-12\\ 6.5299e-11\\ 
		   	3.6820e-11\\ 1.9307e-10\\ 4.1968e-12\\ 1.6521e-11\\ 1.9177e-11\\ 
		   	4.0368e-11\\ 7.7410e-12\\ 5.7960e-11\\ 5.0152e-11\\ };

		\addplot+[only marks] 
			table[header=false, x expr=\lineno+1, y index={0}, row sep=\\]
		   {5.2387e-12\\ 2.1345e-13\\ 6.0076e-11\\ 1.0107e-09\\ 
		   	3.7518e-10\\ 2.1425e-10\\ 3.6459e-13\\ 8.0090e-10\\ 2.0050e-11\\ 
		   	2.8740e-13\\ 6.6507e-12\\ 3.8184e-10\\ 3.4616e-11\\ 3.2844e-16\\ 
		   	2.0789e-13\\ 1.4241e-11\\ 9.0137e-10\\ 2.4530e-12\\ 1.1717e-11\\ 
		   	3.6912e-10\\ 3.4428e-10\\ 1.5802e-11\\ 2.9787e-11\\ 7.1151e-12\\ 
		   	6.9520e-13\\ 5.8205e-12\\ 2.9468e-12\\ 1.3045e-11\\ 
		   	2.8839e-11\\ 4.9737e-11\\ 8.2478e-12\\ 8.7986e-12\\ 8.7175e-11\\ 
		   	1.4825e-11\\ 1.0487e-10\\ 5.7005e-12\\ 2.8869e-11\\ 3.2935e-11\\ 
		   	2.7555e-11\\ 1.4376e-11\\ 3.4975e-11\\ 2.2450e-10\\ };

		\addplot+[only marks] 
			table[header=false, x expr=\lineno+1, y index={0}, row sep=\\]
		   {1.1375e-14\\ 1.0532e-15\\ 2.4955e-16\\ 7.7434e-16\\ 
		   	3.6727e-16\\ 3.5584e-16\\ 3.6559e-16\\ 8.0713e-16\\ 6.6231e-16\\ 
		   	8.7645e-14\\ 5.8446e-16\\ 2.6815e-16\\ 1.1710e-15\\ 1.5966e-14\\ 
		   	5.5414e-16\\ 2.1949e-16\\ 1.7258e-15\\ 6.9903e-16\\ 5.6589e-16\\ 
		   	4.8459e-16\\ 6.0903e-16\\ 2.7423e-16\\ 2.5098e-16\\ 1.2689e-15\\ 
		   	5.6838e-16\\ 4.3476e-16\\ 2.5147e-15\\ 1.5281e-13\\ 
		   	3.1990e-14\\ 5.2243e-16\\ 3.2584e-16\\ 1.9061e-14\\ 6.0716e-16\\ 
		   	5.9279e-16\\ 2.3941e-16\\ 4.3044e-16\\ 5.8900e-16\\ 2.1493e-16\\ 
		   	6.5474e-16\\ 2.3164e-16\\ 6.3299e-16\\ 5.3137e-16\\ };

		\addplot+[only marks] 
			table[header=false, x expr=\lineno+1, y index={0}, row sep=\\]
		   {1.8734e-16\\ 2.0899e-16\\ 2.9221e-16\\ 5.7789e-16\\ 
		   	4.7098e-16\\ 2.0839e-16\\ 2.5847e-16\\ 2.4154e-16\\ 5.4835e-16\\ 
		   	8.7998e-14\\ 4.5515e-16\\ 2.7398e-16\\ 4.3499e-16\\ 3.8836e-16\\ 
		   	3.3601e-16\\ 2.2863e-16\\ 2.8171e-16\\ 3.8490e-16\\ 5.3350e-16\\ 
		   	2.1783e-16\\ 2.2356e-16\\ 1.1829e-15\\ 3.2275e-16\\ 7.6782e-16\\ 
		   	2.9003e-16\\ 3.5672e-16\\ 4.0286e-16\\ 4.0698e-16\\ 
		   	1.9944e-16\\ 2.3347e-16\\ 2.7796e-16\\ 2.7311e-16\\ 3.7614e-16\\ 
		   	2.2041e-16\\ 2.2355e-16\\ 4.2370e-16\\ 8.4558e-16\\ 9.2044e-16\\ 
		   	2.2987e-16\\ 4.1610e-16\\ 5.9667e-16\\ 3.3332e-16\\ };

		\addplot+[only marks] 
			table[header=false, x expr=\lineno+1, y index={0}, row sep=\\]
		   {3.9112e-16\\ 3.3178e-16\\ 3.9654e-16\\ 3.7281e-16\\ 
		   	3.4168e-16\\ 4.2904e-16\\ 3.7321e-16\\ 3.5991e-16\\ 3.4467e-16\\ 
		   	8.7626e-14\\ 3.6711e-16\\ 2.9958e-16\\ 3.3848e-16\\ 3.2844e-16\\ 
		   	3.3769e-16\\ 3.9415e-16\\ 3.6008e-16\\ 3.8328e-16\\ 3.6924e-16\\ 
		   	3.4520e-16\\ 3.4163e-16\\ 3.8935e-16\\ 4.4952e-16\\ 3.7449e-16\\ 
		   	3.4127e-16\\ 3.5257e-16\\ 3.4451e-16\\ 3.3802e-16\\ 
		   	3.8448e-16\\ 3.6458e-16\\ 3.5527e-16\\ 3.5506e-16\\ 3.4648e-16\\ 
		   	3.4527e-16\\ 3.4576e-16\\ 3.3289e-16\\ 3.6582e-16\\ 3.6089e-16\\ 
		   	3.7116e-16\\ 3.6654e-16\\ 3.4754e-16\\ 3.4480e-16\\ };
    \end{semilogyaxis}
\end{tikzpicture}
\caption{Relative error obtained by the three competing algorithms
  when aiming for half, single, and double precision accuracy over a
  range of test problems.}
\label{fig.accuracy_testmats}
\end{figure}

The results of the experiment are shown in
Figure~\ref{fig.accuracy_testmats}.
We can see that all algorithms tend to obtain the desired
relative error,
although all algorithms struggled to obtain double precision accuracy
on problem number $10$ and our algorithm did not obtain double
precision accuracy in a few of the other test cases.
This is not completely unexpected: we use a rather simple
stopping criteria for our iterative method
(examining the relative difference between two iterates)
whilst the KR algorithm has rigorous a priori error
analysis~\cite{KandolfRelton2016}
and the block method is backward stable\cite{alhi11}.
Further experimentation with these test cases showed that allowing our
method to perform a few more iterations allowed us to reach the
desired accuracy,
so we are merely terminating early rather than performing an unstable
computation.

We also see that, especially for single and half precision,
our algorithm was often the most accurate despite the
backward stable nature of the block algorithm.

\begin{figure}[t]
\centering
\tikzsetnextfilename{accuracy_testmatrices_times}
\pgfplotsset{height=0.35\linewidth,width=0.85\linewidth,compat=1.10,every axis/.append style={legend style={/tikz/every even column/.append style={column sep=6pt}}}}
\pgfplotscreateplotcyclelist{list_atm}{%
color_peter1,line width=1pt, mark=o, solid\\
color_peter1,line width=1pt, mark=square, solid\\
color_peter1,line width=1pt, mark=x, solid\\
color_peter2,line width=1pt, mark=o, solid\\
color_peter2,line width=1pt, mark=square, solid\\
color_peter2,line width=1pt, mark=x, solid\\
color_peter3,line width=1pt, mark=o, solid\\
color_peter3,line width=1pt, mark=square, solid\\
color_peter3,line width=1pt, mark=x, solid\\
}

\noindent%
\begin{tikzpicture}[scale=1]%
    \begin{customlegend}[legend style={at={(0,0.65)}}, anchor= north east, 
 		legend columns=6,
 		legend style={/tikz/every even column/.append style={column sep=0.3cm}},
 		legend entries={new Alg., KR Alg., block Alg., half,single,double}]
        \addlegendimage{black,line width=1pt, mark=o, only marks}
        \addlegendimage{black,line width=1pt, mark=square, only marks}
        \addlegendimage{black,line width=1pt, mark=x, only marks}
        \addlegendimage{color_peter1,line width=1pt, solid, sharp plot}
        \addlegendimage{color_peter2,line width=1pt, solid,sharp plot}
        \addlegendimage{color_peter3,line width=1pt, solid,sharp plot}
   	\end{customlegend}
   	\begin{semilogyaxis}[legend style={at={(1,1.15)}}, 
   	anchor= north east, legend columns=5,cycle list name=list_atm, 
   	xmin=0, xmax=43,grid=major, 
   	xlabel={problem}, ylabel={CPU time in $[sec]$},
   	]
		\addplot+[only marks,] 
			table[header=false, x expr=\lineno+1, y index={0}, row sep=\\]
		   {2.7405e-02\\ 1.1750e-02\\ 6.6330e-03\\ 5.5180e-03\\ 2.7030e-03\\ 
		   	1.9540e-03\\ 2.6410e-03\\ 1.8810e-03\\ 2.5670e-03\\ 1.5760e-03\\ 
		   	2.0170e-03\\ 2.0000e-03\\ 1.5470e-03\\ 9.7500e-04\\ 1.0890e-03\\ 
		   	1.0390e-03\\ 9.9200e-04\\ 1.0230e-03\\ 1.0170e-03\\ 1.0120e-03\\ 
		   	9.4100e-04\\ 9.4300e-04\\ 1.3150e-03\\ 9.2400e-04\\ 1.3160e-03\\ 
		   	2.9055e-02\\ 1.3300e-03\\ 8.9100e-04\\ 9.0500e-04\\ 8.6000e-04\\ 
		   	8.1900e-04\\ 8.2200e-04\\ 8.4800e-04\\ 7.8200e-04\\ 1.1610e-03\\ 
		   	7.8100e-04\\ 8.0700e-04\\ 7.7700e-04\\ 7.7700e-04\\ 7.8400e-04\\ 
		   	8.3700e-04\\ 1.2240e-03\\ };

		\addplot+[only marks] 
			table[header=false, x expr=\lineno+1, y index={0}, row sep=\\]
		   {7.7295e-02\\ 2.5331e-02\\ 9.1330e-03\\ 9.2680e-03\\ 3.6300e-03\\ 
		   	3.2660e-03\\ 2.6870e-03\\ 2.5300e-03\\ 2.6230e-03\\ 2.2260e-03\\ 
		   	2.3230e-03\\ 2.3130e-03\\ 2.0240e-03\\ 1.4330e-03\\ 1.5660e-03\\ 
		   	1.4620e-03\\ 1.4170e-03\\ 5.8470e-03\\ 1.3910e-03\\ 1.3940e-03\\ 
		   	1.3470e-03\\ 1.4000e-03\\ 2.4330e-03\\ 1.3330e-03\\ 1.4370e-03\\ 
		   	3.3627e-02\\ 1.6800e-03\\ 1.2950e-03\\ 1.2510e-03\\ 1.2960e-03\\ 
		   	1.4330e-03\\ 1.2340e-03\\ 1.2720e-03\\ 1.2810e-03\\ 1.4420e-03\\ 
		   	1.2930e-03\\ 1.2620e-03\\ 1.2580e-03\\ 1.2980e-03\\ 1.2830e-03\\ 
		   	3.5100e-03\\ 1.4430e-03\\ };

		\addplot+[only marks] 
			table[header=false, x expr=\lineno+1, y index={0}, row sep=\\]
		   {5.8010e-03\\ 4.5210e-03\\ 5.8900e-03\\ 6.8650e-03\\ 1.1620e-03\\ 
		   	1.3850e-03\\ 1.1340e-03\\ 1.3230e-03\\ 9.3200e-04\\ 1.1910e-03\\ 
		   	9.1400e-04\\ 9.8300e-04\\ 1.3050e-03\\ 9.9100e-04\\ 8.9600e-04\\ 
		   	1.0120e-03\\ 9.9200e-04\\ 7.3000e-04\\ 9.7100e-04\\ 9.9300e-04\\ 
		   	1.0500e-03\\ 9.6600e-04\\ 9.6100e-04\\ 9.2800e-04\\ 8.8800e-04\\ 
		   	5.7390e-03\\ 1.7330e-03\\ 9.3100e-04\\ 8.9800e-04\\ 9.2600e-04\\ 
		   	6.4900e-04\\ 9.1300e-04\\ 8.6600e-04\\ 8.6900e-04\\ 8.5100e-04\\ 
		   	9.1100e-04\\ 8.7000e-04\\ 8.8100e-04\\ 8.4000e-04\\ 8.5600e-04\\ 
		   	8.5700e-04\\ 8.8800e-04\\ };

		\addplot+[only marks] 
			table[header=false, x expr=\lineno+1, y index={0}, row sep=\\]
		   {1.1510e-03\\ 1.7990e-03\\ 2.3850e-03\\ 1.7130e-03\\ 2.2650e-03\\ 
		   	7.6200e-04\\ 1.5960e-03\\ 1.6280e-03\\ 2.2020e-03\\ 1.6010e-03\\ 
		   	2.1980e-03\\ 2.2370e-03\\ 1.6690e-03\\ 5.7100e-04\\ 1.1400e-03\\ 
		   	1.7060e-03\\ 1.6490e-03\\ 1.5940e-03\\ 1.1350e-03\\ 1.1470e-03\\ 
		   	1.1590e-03\\ 7.6500e-04\\ 2.1950e-03\\ 1.5940e-03\\ 1.6160e-03\\ 
		   	6.4370e-03\\ 1.2610e-03\\ 7.6600e-04\\ 7.6600e-04\\ 1.6260e-03\\ 
		   	1.6240e-03\\ 1.6180e-03\\ 1.5940e-03\\ 1.5960e-03\\ 1.6130e-03\\ 
		   	1.6000e-03\\ 1.6080e-03\\ 1.5930e-03\\ 1.5940e-03\\ 1.6100e-03\\ 
		   	1.5990e-03\\ 1.6060e-03\\ };

		\addplot+[only marks] 
			table[header=false, x expr=\lineno+1, y index={0}, row sep=\\]
		   {1.8360e-03\\ 1.7970e-03\\ 1.9670e-03\\ 1.9460e-03\\ 1.7260e-03\\ 
		   	1.0880e-03\\ 1.5430e-03\\ 1.5550e-03\\ 1.7080e-03\\ 1.5920e-03\\ 
		   	1.7210e-03\\ 1.8110e-03\\ 1.6670e-03\\ 1.0650e-03\\ 1.4980e-03\\ 
		   	1.5570e-03\\ 1.5370e-03\\ 1.5980e-03\\ 1.4040e-03\\ 1.3840e-03\\ 
		   	1.3600e-03\\ 1.3520e-03\\ 1.9140e-03\\ 1.6170e-03\\ 1.6180e-03\\ 
		   	4.3440e-03\\ 1.5010e-03\\ 1.2890e-03\\ 1.2950e-03\\ 1.6610e-03\\ 
		   	1.6540e-03\\ 1.5900e-03\\ 1.6460e-03\\ 1.6280e-03\\ 1.6030e-03\\ 
		   	1.5920e-03\\ 1.5970e-03\\ 1.6090e-03\\ 1.8000e-03\\ 1.6140e-03\\ 
		   	1.6330e-03\\ 1.6560e-03\\ };

		\addplot+[only marks] 
			table[header=false, x expr=\lineno+1, y index={0}, row sep=\\]
		   {1.1890e-03\\ 9.5200e-04\\ 1.0290e-03\\ 8.5200e-04\\ 5.9000e-04\\ 
		   	8.9400e-04\\ 7.3300e-04\\ 8.5900e-04\\ 5.8000e-04\\ 7.4900e-04\\ 
		   	5.8800e-04\\ 6.5500e-04\\ 8.9000e-04\\ 9.0400e-04\\ 7.5400e-04\\ 
		   	8.7900e-04\\ 8.3700e-04\\ 5.5400e-04\\ 8.4900e-04\\ 8.6700e-04\\ 
		   	9.0400e-04\\ 9.1500e-04\\ 9.5200e-04\\ 8.8200e-04\\ 8.8600e-04\\ 
		   	1.2710e-03\\ 1.1000e-03\\ 8.8600e-04\\ 8.6200e-04\\ 7.4200e-04\\ 
		   	5.5400e-04\\ 9.4800e-04\\ 8.9200e-04\\ 8.8600e-04\\ 8.7000e-04\\ 
		   	8.7700e-04\\ 8.6800e-04\\ 8.8400e-04\\ 8.6300e-04\\ 8.8300e-04\\ 
		   	8.6300e-04\\ 8.8100e-04\\ };

		\addplot+[only marks] 
			table[header=false, x expr=\lineno+1, y index={0}, row sep=\\]
		   {7.2370e-03\\ 3.7340e-03\\ 5.9660e-03\\ 4.7900e-03\\ 5.8810e-03\\ 
		   	7.6500e-04\\ 3.6870e-03\\ 2.8780e-03\\ 7.0840e-03\\ 3.7220e-03\\ 
		   	5.8240e-03\\ 5.8840e-03\\ 4.7710e-03\\ 4.5700e-04\\ 2.8900e-03\\ 
		   	4.6970e-03\\ 4.7540e-03\\ 5.7860e-03\\ 1.5980e-03\\ 7.1120e-03\\ 
		   	2.1780e-03\\ 7.6200e-04\\ 7.0750e-03\\ 3.7660e-03\\ 4.6930e-03\\ 
		   	1.5715e-02\\ 5.9690e-03\\ 1.1430e-03\\ 1.1760e-03\\ 7.0730e-03\\ 
		   	4.6780e-03\\ 4.7420e-03\\ 3.7230e-03\\ 3.7040e-03\\ 3.6990e-03\\ 
		   	3.7380e-03\\ 7.1510e-03\\ 3.7120e-03\\ 3.6960e-03\\ 2.8740e-03\\ 
		   	5.8320e-03\\ 4.6640e-03\\ };

		\addplot+[only marks] 
			table[header=false, x expr=\lineno+1, y index={0}, row sep=\\]
		   {2.4310e-03\\ 2.8550e-03\\ 4.5610e-03\\ 3.4900e-03\\ 3.5170e-03\\ 
		   	1.1900e-03\\ 2.3920e-03\\ 2.0190e-03\\ 3.7590e-03\\ 2.3740e-03\\ 
		   	3.8020e-03\\ 3.8130e-03\\ 2.2920e-03\\ 1.1400e-03\\ 2.1790e-03\\ 
		   	3.3200e-03\\ 2.9120e-03\\ 4.3180e-03\\ 2.3520e-03\\ 2.6970e-03\\ 
		   	2.5750e-03\\ 1.0035e-02\\ 3.3620e-03\\ 3.3250e-03\\ 2.4200e-03\\ 
		   	5.9970e-03\\ 1.8780e-03\\ 1.4870e-03\\ 1.4460e-03\\ 2.8990e-03\\ 
		   	2.3840e-03\\ 2.4740e-03\\ 2.2040e-03\\ 3.0480e-03\\ 2.2510e-03\\ 
		   	3.0790e-03\\ 2.6460e-03\\ 3.0600e-03\\ 2.5750e-03\\ 1.9960e-03\\ 
		   	2.3630e-03\\ 2.5640e-03\\ };

		\addplot+[only marks] 
			table[header=false, x expr=\lineno+1, y index={0}, row sep=\\]
		   {1.2670e-03\\ 1.0310e-03\\ 1.1590e-03\\ 9.8400e-04\\ 6.2300e-04\\ 
		   	1.0060e-03\\ 7.5700e-04\\ 1.0160e-03\\ 6.3300e-04\\ 8.2900e-04\\ 
		   	6.5000e-04\\ 6.5400e-04\\ 1.0130e-03\\ 8.9600e-04\\ 8.0800e-04\\ 
		   	1.0160e-03\\ 9.9500e-04\\ 6.0900e-04\\ 9.7200e-04\\ 1.0000e-03\\ 
		   	1.0460e-03\\ 1.0310e-03\\ 1.0450e-03\\ 1.0150e-03\\ 1.0080e-03\\ 
		   	1.2730e-03\\ 1.1530e-03\\ 9.8400e-04\\ 1.0200e-03\\ 8.2400e-04\\ 
		   	6.0700e-04\\ 9.9300e-04\\ 1.0090e-03\\ 1.0350e-03\\ 1.0070e-03\\ 
		   	1.0530e-03\\ 1.0420e-03\\ 1.0350e-03\\ 1.0270e-03\\ 1.0350e-03\\ 
		   	1.0120e-03\\ 1.0510e-03\\ };
    \end{semilogyaxis}
\end{tikzpicture}
\caption{Time (in seconds) required by the three competing algorithms
  when aiming for half, single, and double precision accuracy over a
  range of test problems.}
\label{fig.timing_testmats}
\end{figure}
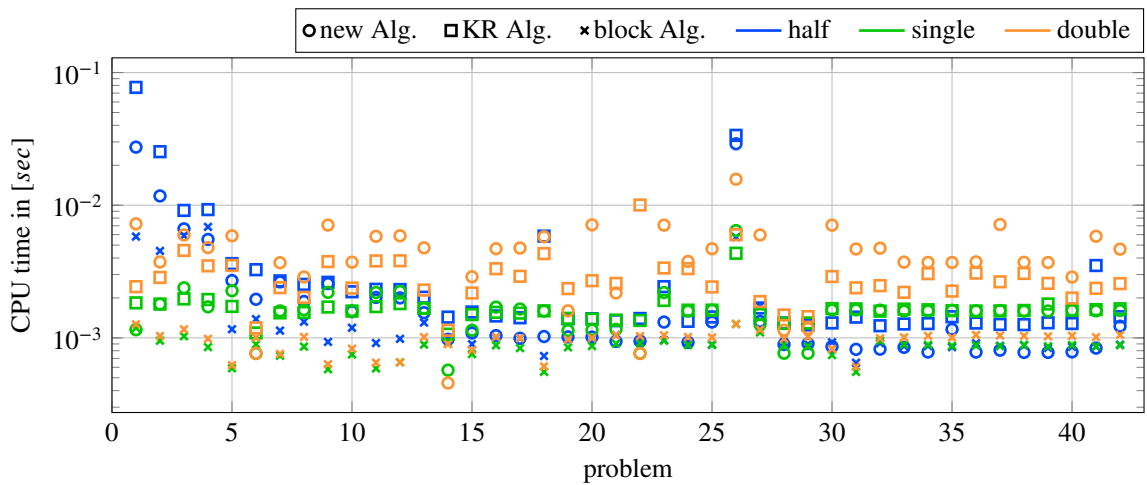

In Figure~\ref{fig.timing_testmats} we plot the time required for each
of the computations performed by all the algorithms.
For half and single precision accuracy we find that the
either our new algorithm or the block algorithm are fastest,
with the KR algorithm trailing behind.
When aiming for double precision accuracy our new method is often
slower than the KR algorithm which, in turn,
is slower than the block algorithm.

This is also to be expected,
both our algorithm and the KR algorithm are based upon
Krylov methods and low-rank approximation,
which tend to work best on large sparse problems as opposed to the
small dense problems considered here.
Furthermore, our algorithm approximates the entire
\Fd\
(using a Krylov space which is independent of the vector)
which is then multiplied by a vector.
Therefore the cost of applying our new method to multiple vectors is
essentially the same as for a single vector.
By contrast the KR algorithm builds a Krylov space dependent upon the
vector and must be entirely rerun should this vector change.
In the next subsection we require the computation of the entire
\Fd,
a situation in which Algorithm~\ref{alg:arnoldi_frechet} excels.

\subsection{Nuclear activation and
  transmutation}\label{sec:nucl-transm}

One application requiring the entire \Fd\ is
the computation of the sensitivity of nuclear activation and
transmutation events.
To brief\/ly summarize,
we are interested in the sensitivity of $\vf^T\!\!\vx(t)$
to perturbations in the matrix $A$,
where $\vx(t)$ and $A$ satisfy the Bateman equation
\begin{equation}\nonumber
  \label{eq.bateman}
  \frac{d\vx}{dt} = A\vx(t),\quad \vx(0) = \bf{x_0}.
\end{equation}
The solution to this equation is clearly $\vx(t) = \exp(tA)\bf{x_0}$.
Within this application $\vx(t)$ gives the time-varying nuclide numbers and
the matrix $A$ (which is sparse and nonsymmetric) contains the
coefficients associated to various nuclear reactions.
Since the elements of the matrix $A$ are determined
via physical experiments they are inherently noisy and it is important
to check that the quantity $\vf^T\!\!\vx(t)$ is not overly sensitive to
perturbations in these values.
In the appendix of~\cite{amrh15} it is shown that the $k$ most
sensitive entries of $A$ are the $k$ largest elements of
\begin{equation}\nonumber
  \label{eq.nuclear_expm}
  L_{\exp}(tA^T, E) = L_{\exp}(tA, E^T)^T,\quad \mbox{where $E = \vf\bf{x_0}^T$}.
\end{equation}

By combining recent work from Higham and Relton~\cite{hire16}
with efficient algorithms for $L_{\exp}(tA, E^T)\vb$ presented in
\cite{KandolfRelton2016} and the previous sections,
the largest $k$ elements of $L_{\exp}(tA, E^T)$ can be found without
forming the entire \Fd\ itself.
However,
one often requires all the sensitivities and therefore needs to
compute the entire \Fd.

We will test the relative error and the time to compute these
\Fd\ using our new methodology,
on three real test problems from nuclear physics,
when compared against the block $2\times 2$ approach and
the code \texttt{expm\_frechet\_pade},
found in the Matrix Function Toolbox~\cite{Higham2008},~\cite{high-mft}.
Since the vectors $\vf$ and $\bf{x_0}$ arising from this application are
not equal we use the Arnoldi algorithm. In each case we simply use the value $t = 1$.

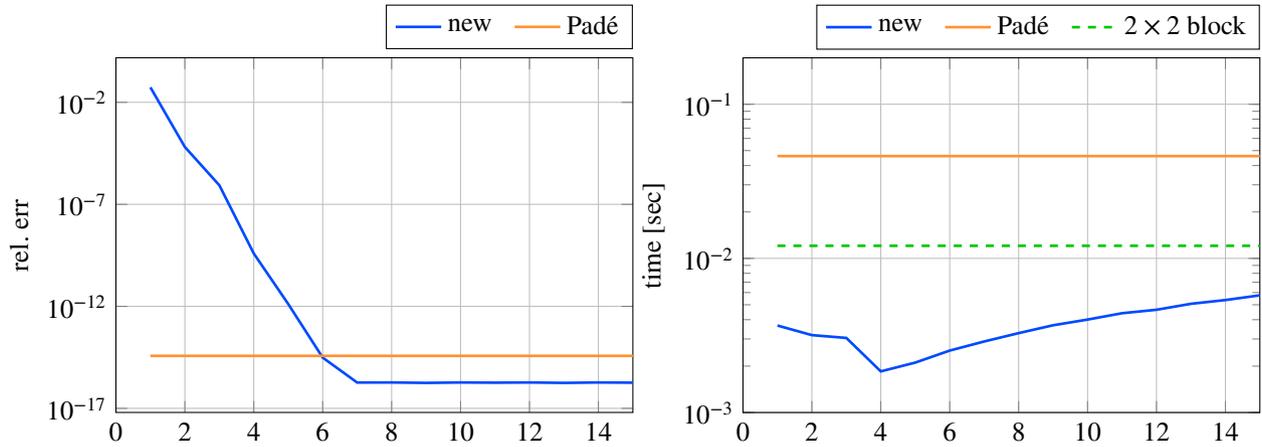
\begin{figure}
  \centering
\tikzsetnextfilename{wmix}
\pgfplotsset{height=0.35\linewidth,width=0.47\linewidth,compat=1.10,every axis/.append style={legend style={/tikz/every even column/.append style={column sep=6pt}}}}
\pgfplotscreateplotcyclelist{list_wmix}{%
color_peter1,line width=1pt, mark=none, solid\\
color_peter3,line width=1pt, mark=none, solid\\
color_peter2,line width=1pt, mark=none, dashed\\
}
\begin{tikzpicture}[scale=1]%
    \begin{semilogyaxis}[legend style={at={(1,1.15)}}, 
   		anchor= north east, legend columns=5,cycle list name=list_wmix, 
   		xmin=0, xmax=15,grid=major, 
   		xlabel={}, ylabel={rel.\ err}]
		\addplot+[] 
			table[header=false, x expr=\lineno+1, y index={0}, row sep=\\]
		   {5.4313e-02\\ 6.4663e-05\\ 8.6120e-07\\ 3.8825e-10\\ 1.3013e-12\\ 
		   	3.0669e-15\\ 1.8384e-16\\ 1.8537e-16\\ 1.7685e-16\\ 1.8537e-16\\ 
		   	1.8111e-16\\ 1.8537e-16\\ 1.7685e-16\\ 1.8537e-16\\ 1.8111e-16\\ };
		\addlegendentry{new};
		\addplot+[] 
			table[header=false, x index={0}, y index={1}, row sep=\\]
		   {1 3.7448e-15\\15 3.7448e-15\\};
		\addlegendentry{Pad\'{e}};
	\end{semilogyaxis}
\end{tikzpicture}%
\begin{tikzpicture}[scale=1]%
    \begin{semilogyaxis}[legend style={at={(1,1.15)}}, 
   		anchor= north east, legend columns=5,cycle list name=list_wmix, 
   		xmin=0, xmax=15,ymin=1e-3, ymax=2e-1,grid=major, 
   		xlabel={}, ylabel={time [sec]}]
		\addplot+[] 
			table[header=false, x expr=\lineno+1, y index={0}, row sep=\\]
		   {3.6680e-03\\ 3.1780e-03\\ 3.0470e-03\\ 1.8490e-03\\ 2.1080e-03\\ 
		   	2.5230e-03\\ 2.8900e-03\\ 3.2740e-03\\ 3.6820e-03\\ 4.0080e-03\\ 
		   	4.4070e-03\\ 4.6410e-03\\ 5.0740e-03\\ 5.3660e-03\\ 5.7580e-03\\ };
		\addlegendentry{new},
		\addplot+[] 
			table[header=false, x index={0}, y index={1}, row sep=\\]
		   {1 4.6082e-02\\15 4.6082e-02\\};		   
		\addlegendentry{Pad\'{e}};
			\addplot+[] 
			table[header=false, x index={0}, y index={1}, row sep=\\]
		   {1 1.2064e-02\\ 15 1.2064e-02\\};
		   \addlegendentry{$2\times 2$ block};
    \end{semilogyaxis}
\end{tikzpicture}
  \caption{The relative error (left) and the time in seconds (right)
    when computing a low-rank approximation to the
    \Fd\ for the \texttt{wmix} matrix.
    The rank of the approximation is given on the $x$-axis.}
  \label{fig.wmix}
\end{figure}

Our first test problem results in a matrix of size $69 \times 69$
named \texttt{wmix}.
In Figure~\ref{fig.wmix} we give the relative error in comparison to
the block method and the time to compute each approximation.
Since the block method and \texttt{expm\_frechet\_pade}
compute the exact \Fd\ instead of a low-rank
approximation their timings are constant as the rank changes.
We see that even a rank 7 approximation is numerically identical
to the solution returned by the block method and is much faster than
the other approaches.

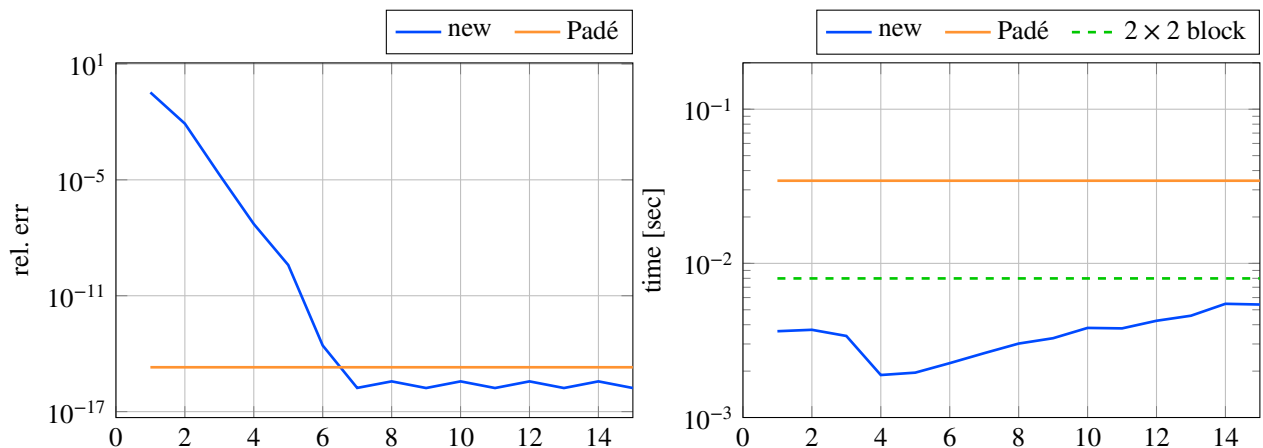
\begin{figure}
  \centering
\tikzsetnextfilename{rwmix}
\pgfplotsset{height=0.35\linewidth,width=0.47\linewidth,compat=1.10,every axis/.append style={legend style={/tikz/every even column/.append style={column sep=6pt}}}}
\pgfplotscreateplotcyclelist{list_rwmix}{%
color_peter1,line width=1pt, mark=none, solid\\
color_peter3,line width=1pt, mark=none, solid\\
color_peter2,line width=1pt, mark=none, dashed\\
}
\begin{tikzpicture}[scale=1]%
    \begin{semilogyaxis}[legend style={at={(1,1.15)}}, 
   		anchor= north east, legend columns=5,cycle list name=list_rwmix, 
   		xmin=0, xmax=15,grid=major, 
   		xlabel={}, ylabel={rel.\ err}]
		\addplot+[] 
			table[header=false, x expr=\lineno+1, y index={0}, row sep=\\]
		   {3.3416e-01\\ 8.1413e-03\\ 1.8880e-05\\ 5.1723e-08\\ 3.8938e-10\\ 
		   	2.6538e-14\\ 1.6882e-16\\ 3.6814e-16\\ 1.6677e-16\\ 3.6803e-16\\ 
		   	1.6677e-16\\ 3.6803e-16\\ 1.6677e-16\\ 3.6803e-16\\ 1.6677e-16\\ };
		\addlegendentry{new};
		\addplot+[] 
			table[header=false, x index={0}, y index={1}, row sep=\\]
		   {1 1.9796e-15\\15 1.9796e-15\\};
		\addlegendentry{Pad\'{e}};
	\end{semilogyaxis}
\end{tikzpicture}%
\begin{tikzpicture}[scale=1]%
    \begin{semilogyaxis}[legend style={at={(1,1.15)}}, 
   		anchor= north east, legend columns=5,cycle list name=list_rwmix, 
   		xmin=0, xmax=15,ymin=1e-3, ymax=2e-1,grid=major, 
   		xlabel={}, ylabel={time [sec]}]
		\addplot+[] 
			table[header=false, x expr=\lineno+1, y index={0}, row sep=\\]
		   {3.6300e-03\\ 3.7080e-03\\ 3.3820e-03\\ 1.8870e-03\\ 1.9550e-03\\ 
		   	2.2530e-03\\ 2.6150e-03\\ 3.0180e-03\\ 3.2700e-03\\ 3.8140e-03\\ 
		   	3.7870e-03\\ 4.2430e-03\\ 4.5760e-03\\ 5.4660e-03\\ 5.4080e-03\\ };
		\addlegendentry{new},
		\addplot+[] 
			table[header=false, x index={0}, y index={1}, row sep=\\]
		   {1 3.4373e-02\\15 3.4373e-02\\};		   
		\addlegendentry{Pad\'{e}};
			\addplot+[] 
			table[header=false, x index={0}, y index={1}, row sep=\\]
		   {1 7.9880e-03\\ 15 7.9880e-03\\};
		   \addlegendentry{$2\times 2$ block};
    \end{semilogyaxis}
\end{tikzpicture}
  \caption{The relative error (left) and the time in seconds (right)
    when computing a low-rank approximation to the
    \Fd\ for the \texttt{rwmix} matrix.
    The rank of the approximation is given on the $x$-axis.}
  \label{fig.rwmix}
\end{figure}
The next problem, \texttt{rwmix}, in Figure~\ref{fig.rwmix}
requires a matrix of size $62 \times 62$.
We see that only a rank 7 approximation is required to obtain full
double precision accuracy and the method is once again much faster
than the alternatives.

\begin{figure}
  \centering
\tikzsetnextfilename{wmixcooler}
\pgfplotsset{height=0.35\linewidth,width=0.47\linewidth,compat=1.10,every axis/.append style={legend style={/tikz/every even column/.append style={column sep=6pt}}}}
\pgfplotscreateplotcyclelist{list_wmixcooler}{%
color_peter1,line width=1pt, mark=none, solid\\
color_peter3,line width=1pt, mark=none, solid\\
color_peter2,line width=1pt, mark=none, dashed\\
}
\begin{tikzpicture}[scale=1]%
    \begin{semilogyaxis}[legend style={at={(1,1.15)}}, 
   		anchor= north east, legend columns=5,cycle list name=list_wmixcooler, 
   		xmin=0, xmax=15,grid=major, 
   		xlabel={}, ylabel={rel.\ err}]
		\addplot+[] 
			table[header=false, x expr=\lineno+1, y index={0}, row sep=\\]
		   {5.4313e-02\\ 6.4663e-05\\ 8.6120e-07\\ 3.8815e-10\\ 1.9320e-14\\ 
		   	1.3583e-16\\ 1.3013e-16\\ 1.2844e-16\\ 1.2586e-16\\ 1.2844e-16\\ 
		   	1.3012e-16\\ 1.2844e-16\\ 1.2586e-16\\ 1.2844e-16\\ 1.3012e-16\\ };
		\addlegendentry{new};
		\addplot+[] 
			table[header=false, x index={0}, y index={1}, row sep=\\]
		   {1 4.6168e-16\\15 4.6168e-16\\};
		\addlegendentry{Pad\'{e}};
	\end{semilogyaxis}
\end{tikzpicture}%
\begin{tikzpicture}[scale=1]%
    \begin{semilogyaxis}[legend style={at={(1,1.15)}}, 
   		anchor= north east, legend columns=5,cycle list name=list_wmixcooler, 
   		xmin=0, xmax=15,ymin=1e-3, ymax=2e-1,grid=major, 
   		xlabel={}, ylabel={time [sec]}]
		\addplot+[] 
			table[header=false, x expr=\lineno+1, y index={0}, row sep=\\]
		   {3.1530e-03\\ 2.8880e-03\\ 2.8310e-03\\ 1.5640e-03\\ 1.7920e-03\\ 
		   	2.0760e-03\\ 2.4210e-03\\ 2.6150e-03\\ 2.8530e-03\\ 3.0510e-03\\ 
		   	3.4810e-03\\ 3.7900e-03\\ 4.0110e-03\\ 4.2230e-03\\ 4.8620e-03\\ };
		\addlegendentry{new},
		\addplot+[] 
			table[header=false, x index={0}, y index={1}, row sep=\\]
		   {1 4.13e-02\\15 4.13e-02\\};		   
		\addlegendentry{Pad\'{e}};
			\addplot+[] 
			table[header=false, x index={0}, y index={1}, row sep=\\]
		   {1 5.3e-03\\ 15 5.3e-03\\};
		   \addlegendentry{$2\times 2$ block};
    \end{semilogyaxis}
\end{tikzpicture}
  \caption{The relative error (left) and the time in seconds (right)
    when computing a low-rank approximation to the
    \Fd\ for the \texttt{rwmixcool} matrix.
    The rank of the approximation is given on the $x$-axis.}
  \label{fig.wmixcool}
\end{figure}
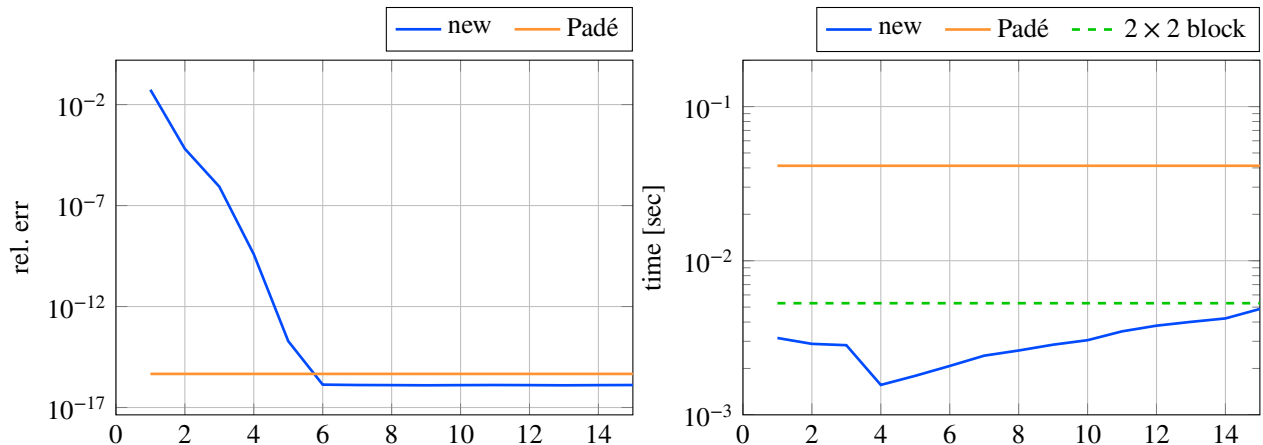

Finally the results for the \texttt{wmixcool} problem
are shown in Figure~\ref{fig.wmixcool}.
This is again a $69 \times 69$ matrix.
As before we see that full double precision accuracy is obtained
by a rank 7 approximation and our new method by
far the least time consuming.

Although we only have access to small examples here,
it is not unusual for matrices in this domain to be have
thousands of rows and columns
and to be rerun for many different time points $t$.
This growth in the matrix size and number of time points
will severely punish the time required to obtain an accurate answer
for the block method and \texttt{expm\_frechet\_pade},
since they treat $A$ as a dense matrix.
By contrast,
using a low-rank approximation will require us only to work
with dense matrices of a much smaller size and save large amounts of
time.

\section{Conclusions}\label{sec:conclusions}

We have presented different Krylov subspace methods for computing low-rank
approximations of the Fr\'{e}chet derivative $L_f(A,E)$ for rank one direction matrices $E$. 
The algorithms are applicable for various properties of $f$, $A$ and $E$: The Lanczos algorithm can be used when both $A$ and $E$ are Hermitian, 
Arnoldi and two-sided Lanczos in the general case
and block Lanczos in case only $A$ is Hermitian. 
We have given methods that are applicable when $f$ is an analytic function or a Stieltjes function,
and separately treated the case of the logarithmic function.
In addition to the standard polynomial versions of these algorithms, we
have also illustrated the use of extended and Krylov subspaces in conjunction with these methods. 
Various a priori convergence results given for all of these functions and Hermitian $A$ illustrate the converge properties of the algorithms.
For analytic $f$ and the Arnoldi approximation, we have proposed a way to carry out a posteriori error estimation and also numerically illustrated the efficiency of the estimate.
Finally, we illustrated the effectiveness of our approaches in comparison to several established methods by various numerical experiments.
We emphasize that all of our algorithms can be generalized to the case of low rank $E$ (i.e., not necessarily rank one), either by linearity of the \Fd\ or by employing block approaches.
We believe that this unified treatment of the problem will help to choose an appropriate numerical method  for approximating the 
Fr\'{e}chet derivative $L_f(A,E)$ when $A$ is sparse and $E$ has a low rank structure.

\ack 
\section*{Acknowledgments}
The authors would like to thank Daniel Kressner for inspiring and fruitful discussions on the topic.

\bibliography{matrixfunctions}

\end{document}